\documentclass[a4paper,11pt,leqno]{amsart}

\usepackage{amsmath,amsfonts,amssymb,amsthm}
\usepackage{bm}
\usepackage{latexsym}

\usepackage{color}
\usepackage{graphicx}
\usepackage{float}

\usepackage[
    bookmarks=true,
    bookmarksnumbered=true,
    bookmarkstype=toc,
    hidelinks]{hyperref}

\usepackage{url}
\usepackage{enumerate}

\newcommand{\R}{\mathbb{R}}

\newcommand{\N}{\mathbb{N}}

\DeclareMathOperator{\supp}{supp}
\DeclareMathOperator{\dist}{dist}

\DeclareMathOperator{\tr}{tr}

\DeclareMathOperator{\USC}{USC}
\DeclareMathOperator{\LSC}{LSC}

\theoremstyle{plain}
\newtheorem{thm}{Theorem}[section]
\newtheorem*{thm*}{Theorem}
\newtheorem{lem}[thm]{Lemma}
\newtheorem*{lem*}{Lemma}

\newtheorem*{cor*}{Corollary}
\newtheorem{prop}[thm]{Proposition}
\newtheorem*{prop*}{Proposition}

\theoremstyle{definition}
\newtheorem{dfn}[thm]{Definition}
\newtheorem*{dfn*}{Definition}

\newtheorem*{ex*}{Example}
\newtheorem{rem}[thm]{Remark}
\newtheorem*{rem*}{Remark}

\newtheorem*{notation*}{Notation}

\numberwithin{equation}{section}

\title[A Game Approach to Free Boundary Problems]
{A Game Approach to Free Boundary Problems of 
Anisotropic Forced Mean Curvature Flow Equations}
\author{Takuya Sato}
\address{(Takuya Sato) Graduate School of Mathematical Sciences, 
University of Tokyo 3-8-1 Komaba, 
Meguro-ku, Tokyo, 153-8914, Japan}
\email{satoh-t@g.ecc.u-tokyo.ac.jp}

\begin{document}

\maketitle

\begin{abstract}
We consider the free boundary problems of degenerate elliptic equations
that describe the level set formulation of 
the interface motion evolved
by anisotropic forced mean curvature flows.
The type of free boundary problems in this paper 
was initially studied as the first-order Hamilton-Jacobi-Isaacs equations 
arising in pursuit-evasion differential games 
and applied to the models of first-order front propagation
in \cite{MR1279982}.
In this paper, 
we consider an extension of these free boundary problems 
to the second-order equations 
and give a deterministic game representation 
based on a discrete approximation scheme in \cite{MR2200259}.
Furthermore, we prove the comparison principle 
for our free boundary problems
by using the framework of time-discrete games.
\end{abstract}

\tableofcontents

\section{Introduction}
We study the Dirichlet problem with free boundary
of the form
\begin{equation}\label{FBP}
  \left\{
  \begin{aligned}
  F(DU,D^2U) & = 1 &&\text{in } D_t\setminus \overline{D}_0,\\
  U &= G &&\text{on } \partial D_0,\\
  U &< t &&\text{in } D_t\setminus \overline{D}_0,\\
  U(x) &\to t &&\text{as } x\to x_0\in \partial D_t,
  \end{aligned}
  \right.\tag{FBP$_t$}
\end{equation}
where $F:(\R^n\setminus \{0\})\times \mathbb{S}^n\to \R$ is a given degenerate elliptic 
and geometric function, 
$D_0\subset \R^n$ is a given open set, 
$G\in C(\partial D_0)$ is a given function
and $t\in (0,+\infty]$ is a parameter.
Here, a pair $(D_t,U)$ of an open set and a function 
that satisfy the above four conditions is called a solution of (\ref{FBP}).
That is, 
we consider $\partial D_t$ to be a free boundary, 
whose determination is part of the problem to be solved.

The first purpose of this paper is to give a game interpretation of (\ref{FBP})
as an extension of \cite{MR2200259}.
We identify the family of deterministic, discrete time, 
two-person games such that the limit of its value functions 
gives a solution of (\ref{FBP}).
The second purpose is to show the comparison principle
of (\ref{FBP}) by applying our game interpretation.
Our comparison principle shows 
the uniqueness of solutions of (\ref{FBP}).

Throughout the paper, 
we assume that $F$ is given by
\begin{equation*}
  F(p,X)=-\tr\left( \sigma(p){}^t\!\sigma (p)X \right)
  + c(p),
\end{equation*}
where the functions $\sigma\in C(\R^n\setminus\{0\};\mathbb{M}^{n\times m})$ and $c\in C(\R^n)$ satisfy 
the following conditions:
\begin{enumerate}[(\text{A}1)]
  \item $\sigma(\lambda p)=\sigma(p)$, $c(\lambda p)=\lambda c(p)$, $c(p)\geq 0$ for each 
  $p\in \R^n\setminus\{0\}$ and $\lambda>0$.\label{assum:homogenity}
  \item $\sigma|_{\partial B(0,1)}, c|_{\partial B(0,1)}$ is Lipschitz continuous.
  \label{assum:Lipschitz-conti}
  \item $\sigma(-p)=\sigma(p)$, $c(-p)=c(p)$ for each 
  $p\in \R^n\setminus\{0\}$.\label{assum:evenness}
  \item $m\geq n-1$ and $\mathop{\text{Im}}\sigma(p):=\{\sigma(p)w\in\R^n \mid w\in\R^m\}=\langle p \rangle^\perp$ 
  for each $p\in \R^n \setminus\{0\}$.\label{assum:rank}
\end{enumerate}
Here, we denote by $\mathbb{M}^{n\times m}$ 
the set of $n\times m$ matrices and 
by $B(x,r)$ the open ball in $\R^n$ with center $x\in\R^n$, radius $r>0$.
We note that $\sigma$ is positively 0-homogenous and 
$c$ is positively 1-homogenous in assumption (A\ref{assum:homogenity}).
According to the form of $F$ and the above conditions, 
we can easily confirm the ellipticity, that is,
\begin{equation}
  F(p,X)\geq F(p,Y) \quad \text{when } X\leq Y \text{ for every } p,
\end{equation}
and the geometricity, that is,
\begin{equation}\label{geometricity}
  F(\lambda p , \lambda X + \mu p\otimes p) = \lambda F(p,X)
  \quad \text{for every } (p,X),\ \lambda>0 \text{ and } \mu\in\R,
\end{equation}
and also
\begin{equation}
  F_*(0,O):=\liminf_{\substack{p\to 0 \\ X\to O}} F(p,X)
  =F^*(0,O):=\limsup_{\substack{p\to 0 \\ X\to O}} F(p,X)
  =0.
\end{equation}

A typical example of our equation $F(DU,D^2U)=1$ 
with (A\ref{assum:homogenity})-(A\ref{assum:rank}) is 
\begin{align}\label{eq:MCF-level-set}
  -\tr \left[\left(I-\frac{DU\otimes DU}{|DU|^2}\right)D^2U \right]=1,
\end{align}
where $\sigma(p):=I-\frac{p\otimes p}{|p|^2}$ and $c(p):=0$.
This equation is called the level set mean curvature flow equation 
and was first analytically studied by \cite{MR1100211,MR1100206}.
In this case, roughly speaking,
the family of level sets $\{\{x\mid U(x)=s\}\}_{s\in[0,t)}$ 
for a solution $U$ of (\ref{eq:MCF-level-set})
is a solution of the surface evolution equation 
\begin{align*}
  V=-\kappa \quad \text{on } \Gamma_s.
\end{align*}
Here, we assume the moving front $\Gamma_s$ is a boundary of 
an open set $D_s$,
and we denote by $V$ and $\kappa$, respectively, 
the velocity of $\Gamma_s$ along its outward normal direction to $D_s$
and the mean curvature of $\Gamma_s$, 
which is the sum of all principal curvatures.
Moreover, we have in our mind the level set equations of 
motion by ``anisotropic mean curvature with anisotropic outer force''
such as 
\begin{align}\label{eq:CZ99}
  V=-b(\mathbf{n})\kappa + c(\mathbf{n}) \quad \text{on } \Gamma_s,
\end{align}
where $\mathbf{n}$ is the outward normal vector to $\Gamma_s$
and $b,c:\partial B(0,1)\to \R$ are the given functions.
In this case, the level set formulation of (\ref{eq:CZ99}) is given by
\begin{align}\label{eq:CZ99-level-set}
  -b\left(\frac{DU}{|DU|}\right)
  \tr \left[ \left(I-\frac{DU\otimes DU}{|DU|^2}\right)D^2U \right]
  +|DU|c\left(\frac{DU}{|DU|}\right)=1
\end{align}
and (\ref{eq:CZ99-level-set}) satisfies (A\ref{assum:homogenity})-(A\ref{assum:rank})
if $b,c:\partial B(0,1)\to \R$ satisfy 
\begin{itemize}
  \item $b(\mathbf{n})>0$, $c(\mathbf{n})\geq 0$ for all $\mathbf{n}\in \partial B(0,1)$,
  \item $\sqrt{b}$ and $c$ is Lipschitz on $\partial B(0,1)$,
  \item $b(-\mathbf{n})=b(\mathbf{n})$, 
  $c(-\mathbf{n})=c(\mathbf{n})$ for all $\mathbf{n}\in \partial B(0,1)$,
\end{itemize}
by setting $\sigma (p):=\sqrt{b(p/|p|)}\left(I-\frac{p\otimes p}{|p|^2}\right)$
and $c(p):=|p|c(p/|p|)$.
Surface evolution equations of the form 
(\ref{eq:CZ99}) were studied in relation to the thermodynamics of two-phase materials
by S. B. Angenent and M. E. Gurtin 
in \cite{MR1017288,MR918798,MR1013461,MR1256146}.
In the case of planar curves, the asymptotic behavior of the interface was studied
using the Wulff shape determined by 
$b$ and $c$ in \cite{MR1682990}.

First-order cases of free boundary problems of the form (\ref{FBP})
were studied in \cite{MR991958,MR1214756}
as firsrt-order Hamilton-Jacobi-Isaacs equations
arising in pursuit-evasion differential games.
These free boundary problems were used in \cite{MR1279982}
to describe the first-order front propagation of the form 
\begin{align}\label{eq:first-order}
  V=c(x,\mathbf{n}) \quad \text{on }\Gamma_s
\end{align}
with a level set equation and a corresponding differential game.
It was shown that the solution of the level set equation of (\ref{eq:first-order})
can be characterized by a value function of a pursuit-evasion differential game.
The large time asymptotics of the front 
was also studied in \cite{MR1279982} by making use of 
the theory of differential games.
They showed that when $c(x,\mathbf{n})=c(\mathbf{n})$ and $c(\mathbf{n})>0$
for all $|\mathbf{n}|=1$,
the convergence
\begin{align}\label{eq:conv-Wulff}
  \frac{1}{s} \Gamma_s \to \partial (\mathop{\mathrm{Wulff}}(c))\quad \text{as } s\to\infty
\end{align}
holds in the sense of Hausdorff distance,
where we denote by $\mathop{\mathrm{Wulff}}(c)\subset \R^n$ the Wulff shape of $c$.
It is also known that the convergence (\ref{eq:conv-Wulff})
holds for surfaces $\Gamma_s$ satisfying equation (\ref{eq:CZ99})
when the initial surface $\Gamma_0$ encloses 
a sufficiently large ball, 
and this was studied in \cite{MR1204331,MR1674750} by using the level set method.

In this paper, 
we consider a problem that extends the model of \cite{MR1279982} 
by adding a curvature term.
R. V. Kohn and S. Serfaty introduced in \cite{MR2200259}
a discrete approoximation scheme 
for the mean curvature flow equations 
using a deterministic, discrete time game.
Our approach is inspired 
by the results and its differential game approach 
of \cite{MR991958,MR1214756,MR1279982}
and the game-theoritic approximation of \cite{MR2200259}.\vspace{1ex}\\
\textbf{Main results.}
We first show the existence of a function $U^\varepsilon:\R^n\setminus\overline{D}_0\to\R\cup\{+\infty\}$
that satisfies the equation
\begin{align*}
  U^\varepsilon (x)
  =\inf_{(\mathbf{v},\mathbf{w})\in \mathcal{D}}\sup_{\mathbf{b}\in\mathcal{S}}
  \left\{
  \begin{aligned}
    &\varepsilon^2 +U^\varepsilon (x+\delta^\varepsilon(\mathbf{v},\mathbf{w},\mathbf{b}))
    &&\text{if}\ x+\delta^\varepsilon(\mathbf{v},\mathbf{w},\mathbf{b})\in \R^n\setminus \overline{D}_0, \\
    &\varepsilon^2 +G(x+\delta^\varepsilon(\mathbf{v},\mathbf{w},\mathbf{b}))
    &&\text{if}\ x+\delta^\varepsilon(\mathbf{v},\mathbf{w},\mathbf{b})\in \overline{D}_0,
  \end{aligned}
  \right.
\end{align*}
which is called the dynamic programming principle,
and prove that its half-relaxed limits
\begin{align*}
  \overline{U}(x):= \limsup_{\substack{\R^n\setminus \overline{D}_0\ni y\to x \\ \varepsilon \to 0}}
  U^\varepsilon(y)\quad \text{and}\quad 
  \underline{U}(x):=\liminf_{\substack{\R^n\setminus \overline{D}_0\ni y\to x \\ \varepsilon \to 0}}
  U^\varepsilon(y),
\end{align*}
respectively, is a viscosity subsolution and supersolution of 
\begin{align*}
  \left\{
  \begin{aligned}
  F(DU,D^2U) & = 1 &&\text{in } \{x\mid \overline{U}(x)<\infty\}\setminus \overline{D}_0,\\
  U &= G &&\text{on } \partial D_0.
  \end{aligned}
  \right.
\end{align*}
Here, the sets $\mathcal{D}$ and $\mathcal{S}$ are defined as 
\begin{align*}
  \begin{aligned}
    \mathcal{D} &:= \mathcal{D}_1 \times \mathcal{D}_2,\\
    \mathcal{D}_1&:=\{ \mathbf{v}=(v^1,v^2)\in \R^n\times\R^n \mid |v^1|=|v^2|=1 \},\\ 
    \mathcal{D}_2&:=\{ \mathbf{w}=(w^1,\ldots, w^m)\in \R^{m\times m} \mid \{w^i\} \text{: orthonormal basis of }\R^m\},\\
    \mathcal{S}&:=\{ \mathbf{b}=(b^1,\ldots,b^m)\in \{\pm 1\}^{m} \}.
  \end{aligned}
\end{align*}
and a function $\delta^\varepsilon : \mathcal{D}\times\mathcal{S}\to \R^n$ is defined as 
\begin{align*}
  \delta^\varepsilon(\mathbf{v},\mathbf{w},\mathbf{b})
  :=\varepsilon \sqrt{2}\sum_{i=1}^m b^i\sigma(v^1)w^i + \varepsilon^2 c(v^1)v^2.
\end{align*}
These sets and a function $U^\varepsilon$ 
can be interpreted as strategy sets and a value function 
of a deterministic two-person game. 
The details of this interpretation is provided in Section 3.
Furthermore, using this value function and its dynamic programming principle, 
we establish the comparison principle for (\ref{FBP}).
\begin{thm}\label{thm:main}
  Assume $(\mathrm{A}\ref{assum:homogenity})$-$(\mathrm{A}\ref{assum:rank})$ 
  and that $D_0$ is an open set with compact boundary and $G\in C(\partial D_0)$.
  \begin{enumerate}[$(1)$]
    \item Let $(D,W)$ be a viscosity subsolution of $(\mathrm{FBP}_t)$ 
    satisfying the following conditions:
    \begin{itemize}
      \item $D\setminus \overline{D}_0$ is bounded and satisfies $(\mathrm{A}\ref{assum:interior-cone})$.
      \item $W$ is bounded from below on $D\setminus D_0$ and
      continuous on a neighborhood of $\partial(D\setminus D_0)$.
    \end{itemize}
    Then, $D\supset\{\underline{U}<t\}\cup D_0$ and $W\leq \underline{U}$ on $D \setminus D_0$.
    \item Let $(\widetilde{D},V)$ be a viscosity supersolution of $(\mathrm{FBP}_t)$ 
    satisfying the following contiditons:
    \begin{itemize}
      \item $\widetilde{D}\setminus \overline{D}_0$ is bounded and satisfies $(\mathrm{A}\ref{assum:interior-cone})$.
      \item $V$ is bounded from below on $\widetilde{D}\setminus D_0$ 
      and continuous on a neighborhood of $\partial(\widetilde{D}\setminus D_0)$.
    \end{itemize}
    Then, $\widetilde{D}\subset\{\overline{U}<t\}\cup D_0$ and $V\geq \overline{U}$ on $\widetilde{D}\setminus D_0$.
  \end{enumerate}
\end{thm}
Here, condition (A\ref{assum:interior-cone}) is provided in Section 2,
that is a regularity assumption for a boundary of a domain.
The formal definition of viscosity solutions for (\ref{FBP}) is also provided in Section 2.
From this theorem, together with the inequality $\underline{U}\leq\overline{U}$,
the following proposition follows directly.
\begin{thm}
  If $(D,W)$ is a viscosity subsolution of $(\mathrm{FBP}_t)$ and 
  $(\widetilde{D},V)$ is a viscosity supersolution of $(\mathrm{FBP}_t)$,
  and they satisfy the assumpsitons in Theorem $\ref{thm:main}$, 
  then $D\supset \widetilde{D}$ and $W\leq V$ on $\widetilde{D}\setminus D_0$. 
\end{thm}

We use the theory of viscosity solutions in our study.
For a general theory of viscosity solutions, 
see \cite{MR1118699,MR2084272}.
For the viscosity solution theory of Hamilton-Jacobi equations,
particularly the relation between optimal control, 
differential games, and viscosity solutions, 
refer to \cite{MR4328923}.
Many examples and methods of viscosity solutions 
for geometric flow equations are presented in \cite{MR2238463}.

We conclude this introduction by giving some works related to
the study of this paper.
R. V. Kohn and S. Serfaty provided in \cite{MR2681474} 
a deterministic game interpretation for 
general non-singular fully nonlinear elliptic and parabolic PDEs.
Based on the method in \cite{MR2200259}, 
alternative proofs
of various properties of the level set mean curvature flow equations
were provided in \cite{MR2873194,MR3466458}
via a game-theoretic approach,
without relying on advanced partial differential equation theory.
Additionally, 
for obstacle problems of surface evolution equations, 
in particular for their large time behavior,
many results were obtained
by using deterministic games 
in \cite{info:hdl/2115/87822,arXiv:2409.06855}.
As a viscosity solution approach to free boundary problems,
viscosity solutions for Stefan and Hele-Shaw type problems, 
which are different types of free boundary problems from those considered in this paper,
were studied. 
A viscosity solution approach to one-phase Stefan and Hele-Shaw problems were introduced in \cite{MR1994745} 
and two-phase problems were introduced in \cite{MR2763347}, 
and existence and uniqueness results were proved.
\vspace{1ex}\\
\textbf{Organization.}
This paper is organized as follows:
in Section 2, we define viscosity solutions for our free boundary problems
and show the comparison principle for Dirichlet problems
with a generalized boundary condition, 
that we need in our proof of comparison results for (\ref{FBP}).
In Section 3, we define the discrete time, two-person games
and show that the existence of value functions and 
their half relaxed limits are viscosity sub and supersolutions
of our problems.
Finally, we give the proof of our comparison results for
free boundary problems in Section 4.

\section{Preliminaries}

First, we define the solution of (\ref{FBP}).
In our analysis, we use the framework of the viscosity solutions 
for singular equations (see \cite{MR1100211,MR1100206}).
We define solutions to (\ref{FBP}) 
in the same manner as the solution to free boundary problems of 
Hamilton-Jacobi-Isaacs equations of pursuit-evasion problems 
in \cite{MR991958,MR1214756}.
\begin{dfn}[Viscosity solutions of free boundary problems]
Let $D_0\subset \R^n$ be a given open set,
$G\in C(\partial D_0)$ be a given function and 
$t\in(0,+\infty]$ be a parameter. We say that
\begin{enumerate}[$(1)$]
  \item the pair $(D,W)$ of an open set with 
  $D\supset\overline{D}_0$ and a function $W\in\USC(D\setminus D_0)$ 
  is a viscosity subsolution of (\ref{FBP}) 
  if $W$ satisfies 
  \begin{align*}
    \left\{ 
    \begin{aligned}
      F(DW,D^2W) &\leq 1 \quad &&\text{in } D\setminus \overline{D}_0,\\
      W&\leq G \quad &&\text{on }\partial D_0
    \end{aligned}
    \right.
  \end{align*}
  in the sense of viscosity solutions
  and free boundary conditions
  \begin{equation*}
    \left\{
      \begin{aligned}
        W&<t &&\text{in } D\setminus \overline{D}_0,\\
        W(x)&\to t &&\text{as } x\to x_0 \text{ for all } x_0\in\partial D;
      \end{aligned}
    \right.
  \end{equation*}
  \item  the pair $(\widetilde{D},V)$ of an open set with 
  $\widetilde{D}\supset\overline{D}_0$ and a function $V\in\LSC(\widetilde{D}\setminus D_0)$ 
  is a viscosity supersolution of (\ref{FBP}) if
  $V$ satisfies 
  \begin{align*}
    \left\{
    \begin{aligned}
      F(DV,D^2V) &\geq 1 \quad 
      &&\text{in } \widetilde{D}\setminus \overline{D}_0,\\
      V&\geq G \quad 
      &&\text{on }\partial D_0
    \end{aligned}
    \right.
  \end{align*}
  in the sense of viscosity solutions
  and free boundary conditions 
  \begin{equation*}
    \left\{
      \begin{aligned}
        V&<t &&\text{in } \widetilde{D}\setminus \overline{D}_0,\\
        V(x)&\to t &&\text{as } x\to x_0 \text{ for all } x_0\in\partial \widetilde{D};
      \end{aligned}
    \right.
  \end{equation*}
  \item the pair $(D,W)$ is a viscosity solution of (\ref{FBP}) if $(D,W)$ is 
  both a viscosity sub and supersolution of (\ref{FBP}).
\end{enumerate}
\end{dfn}

\begin{rem}
We notice that the boundary conditions
`` $W\leq G$ on $\partial D$ '' in (1) 
and `` $V\geq G$ on $\partial \widetilde{D}$ '' in (2) 
are understood in the sense of viscosity solutions.
\end{rem}

\begin{rem}
If $\partial D$ (resp., $\partial\widetilde{D}$) is empty, 
then we impose no conditions for the limit of $W(x)$ 
as $x\to x_0\in \partial D$ (resp., $V(x)$ as $x\to x_0 \in\partial \widetilde{D}$).
\end{rem}

In Section 3, 
we use the change of variables for our equations called  the Cole-Hopf transformation
or the Kruzhkov transformation.
We denote by $\psi:\R\to\R$ the function defined by 
\begin{align}\label{dfn:psi}
  \psi(r)=1-e^{-r}.
\end{align}
If a smooth function $W$ satisfies $DW\neq 0$ and
$F(DW,D^2W)=1$ in an open set of $\R^n$,
then a function $w:=\psi\circ W$ satisfies
\begin{align}\label{eq:differentiate-W}
  DW = \frac{Dw}{1-w},\quad 
  D^2 W = \frac{D^2 w}{1-w}+\frac{Dw\otimes Dw}{(1-w)^2}
\end{align}
by the chain rule.
Substituting (\ref{eq:differentiate-W}) into 
$F(DW,D^2W)=1$ and applying (\ref{geometricity}),
we have that $w$ satisfies the partial differential equation:
\begin{align}\label{eq:equation-w}
  w+F(Dw,D^2w)=1.
\end{align}
As seen in Proposition \ref{cor:trans},
this calculation can be justified in the sense of viscosity solutions,
and we refer to \cite[Lemma 5]{MR1001919} for its proof.

We are interested in the Dirichlet problem of (\ref{eq:equation-w}):
\begin{equation}\label{eq:Dirichlet-problem}
  \left\{
  \begin{aligned}
    w+F(Dw,D^2w)&=1 &&\text{in } \Omega,\\
    w&=g &&\text{on } \partial\Omega,
  \end{aligned}
  \right.\tag{DP}
\end{equation}
where $\Omega\subset\R^n$ is open and $g\in C(\partial \Omega)$.
We need some properties for solutions of (\ref{eq:Dirichlet-problem}) in the proof of our main results.
\begin{prop}\label{cor:trans}
Assume $g<1$ on $\partial \Omega$ and let $\psi$ be the function given by $(\ref{dfn:psi})$.
If $w\in\USC(\overline{\Omega})$ is a viscosity subsolution of $\mathrm{(\ref{eq:Dirichlet-problem})}$
and satisfies $w\leq 1$ on $\overline{\Omega}$,
then $W(x):=\psi^{-1}(w(x))$ is upper semicontinuous and satisfies
\begin{align*}
  \left\{
  \begin{aligned}
    F(DW,D^2W)&\leq 1 &&\text{in } \Omega\cap \{ W <\infty\},\\
    W&\leq \psi^{-1}\circ g &&\text{on } \partial\Omega \cap \{ W <\infty\}
  \end{aligned}
  \right.
\end{align*}
in the sense of viscosity solutions.
Similarly, if $v\in\LSC(\overline{\Omega})$ is a viscosity supersolution of $\mathrm{(\ref{eq:Dirichlet-problem})}$
and satisfies $v\leq 1$ on $\overline{\Omega}$, 
then $V(x):=\psi^{-1}(v(x))$ is lower semicontinuous and satisfies
\begin{align*}
  \left\{
  \begin{aligned}
    F(DV,D^2V)&\geq 1 &&\text{in } \Omega\cap \mathop{\mathrm{int}}(\{ V <\infty\}),\\
    V&\geq \psi^{-1}\circ g &&\text{on } \partial\Omega \cap \mathop{\mathrm{int}}(\{ V <\infty\})
  \end{aligned}
  \right.
\end{align*}
in the sense of viscosity solutions.
Here, we define $\psi^{-1}(1):=+\infty$.
\end{prop}

Next,
we prove a comparison principle for (\ref{eq:Dirichlet-problem}). 
Theorem \ref{thm:comparison-Dirichlet} is inspired by \cite[Theorem 3.1]{MR1702961}
to extend to singular equations.
We use a simular doubling variable technique 
as \cite{MR1702961,MR3365833}.

We need the following assumption on the regularity of $\partial\Omega$.
\begin{enumerate}[(\text{A}1)]
  \setcounter{enumi}{4}
  \setcounter{enumi}{4}
  \item There exist constants $K,\lambda_0>0$, 
  $\R^n$-neighborhood $\mathcal{W}\supset \partial \Omega$ and
  a bounded uniformly continuous function 
  $\mathbf{n} : \mathcal{W}\to \R^n$ 
  satisfying\label{assum:interior-cone}
  \begin{align*}
    B(x-\lambda \mathbf{n}(x),K \lambda ) \subset \Omega \quad 
    \text{for all } x\in \overline{\Omega}\cap \mathcal{W}\text{ and }
    \lambda\in (0,\lambda_0].
  \end{align*}
\end{enumerate}

\begin{thm}\label{thm:comparison-Dirichlet}
Assume $(\mathrm{A}\ref{assum:homogenity})$, $(\mathrm{A}\ref{assum:Lipschitz-conti})$ 
and $(\mathrm{A}\ref{assum:interior-cone})$. 
Let $w\in \USC(\overline{\Omega})$ and $v\in \LSC(\overline{\Omega})$ 
be a bounded viscosity sub and supersolution of $\mathrm{(\ref{eq:Dirichlet-problem})}$, 
respectively, and assume that either
\begin{equation}\label{assum:w-and-v-1}
  w(x) = \limsup_{\substack{y\to x\\y\in\Omega}}w(y) \text{ for all } x\in\partial\Omega
  \text{ and } v \text{ is continuous on a neighborhood of } \partial \Omega
\end{equation}
or
\begin{equation}\label{assum:w-and-v-2}
  w\text{ is continuous on a neighborhood of }\partial \Omega 
  \text{ and }v(x) = \liminf_{\substack{y\to x\\y\in\Omega}}v(y) \text{ for all } x\in\partial\Omega
\end{equation}
and that $\Omega$ is bounded or $\Omega$ is unbounded and 
$\lim_{|x|\to \infty}w(x)=\lim_{|x|\to\infty}v(x)=\mu$
for a constant $\mu\in\R$.
Then, $w\leq v$ on $\overline{\Omega}$.
\end{thm}

To clarify the proof of Theorem \ref{thm:comparison-Dirichlet},
we describe the structural condition satisfied 
by the function $F$ in our setting.

\begin{prop}\label{lem-CP}
Assume $(\mathrm{A}\ref{assum:homogenity})$ and $(\mathrm{A}\ref{assum:Lipschitz-conti})$.
For all $K>0$, there exists a constant $C>0$ such that 
for all $\rho,\varepsilon >0$, 
if $p,q \in\R^n\setminus\{0\}$ and 
$X,Y\in \mathbb{S}^n$ satisfy
\begin{align}\label{condi:lem-CP}
  &|p-q|\leq K\varepsilon (|p|\wedge|q|),\quad 
  \begin{pmatrix}
    X & O \\
    O & -Y 
  \end{pmatrix}
  \leq \frac{\rho}{\varepsilon^2}
  \begin{pmatrix}
    I_n & -I_n\\
    -I_n & I_n
  \end{pmatrix}
  +\rho I_{2n},
\end{align}
then
\begin{align*}
  F(q,Y)-F(p,X) \leq C(\rho + \varepsilon(|p|\wedge|q|)).
\end{align*}
\end{prop}

\begin{proof}
Let $L_\sigma$ and $L_c$ be the Lipschitz constants of $\sigma$ and $c$ 
determined by (A\ref{assum:Lipschitz-conti}), respectively. 
Take $K,\rho,\varepsilon >0$ and $p,q,X,Y$ so that they satisfy (\ref{condi:lem-CP}).
Then,
\begin{align*}
  &F(q,Y)-F(p,X)
  =\tr[{}^t\!\sigma(p)X\sigma(p)]-\tr[{}^t\!\sigma(q)Y\sigma(q)]
  +c(q)-c(p)\\
  &=\tr \left[
  \begin{pmatrix}
    {}^t\!\sigma(p) & {}^t\!\sigma(q)
  \end{pmatrix}
  \begin{pmatrix}
    X & O \\
    O & -Y
  \end{pmatrix}
  \begin{pmatrix}
    \sigma(p)\\
    \sigma(q)
  \end{pmatrix}
  \right] + c(q)-c(p)\\
  &\leq \frac{\rho}{\varepsilon^2}
  \tr\left[
    {}^t\!(\sigma(p)-\sigma(q))(\sigma(p)-\sigma(q))
  \right]
  +\rho \tr[{}^t\!\sigma(p)\sigma(p)]+\rho \tr[{}^t\!\sigma(q)\sigma(q)]
  +L_c|p-q|\\
  &=\frac{\rho}{\varepsilon^2}\sum_{i=1}^n\sum_{j=1}^m (\sigma_{ij}(p)-\sigma_{ij}(q))^2
  +\rho\sum_{i=1}^n\sum_{j=1}^m (\sigma_{ij}(p)^2+\sigma_{ij}(q)^2)
  +L_c|p-q|\\
  &\leq \frac{\rho }{\varepsilon^2}C_{n,m}L_\sigma^2\left(\frac{|p-q|}{|p|\wedge|q|}\right)^2
  +\rho C_{n,m}\max_{\partial B(0,1)}\lVert\sigma\rVert^2 + L_c|p-q|\\
  &\leq C_{n,m,K,\sigma,c}(\rho +\varepsilon (|p|\wedge|q|)).
\end{align*}
\end{proof}

\begin{proof}[Proof of Theorem $\ref{thm:comparison-Dirichlet}$]
The proof is quite similar to that in \cite[Theorem 3.1]{MR1702961}
and \cite[Appendix A]{MR3365833}, 
and we give a proof to make our paper self-contained.
We prove the statement only under the assumption (\ref{assum:w-and-v-1}).
For the case of (\ref{assum:w-and-v-2}), 
it suffices to consider the situation where $u$ and $v$ are replaced 
by $-v$ and $-u$, respectively.

We argue by contradiction assuming that 
$M:=\sup_{\overline{\Omega}}(w-v)>0$.
From the assumption on the values of $w$ and $v$ at infinity,
there exists $x_0\in\overline{\Omega}$ such that $w(x_0)-v(x_0)=M$.
We define the set $\mathcal{M}\subset \overline{\Omega}$ by
\begin{align*}
  \mathcal{M}=\{x\in\overline{\Omega}\mid w(x)-v(x)=M\}
\end{align*}
and define the subsets $\Gamma_w,\Gamma_v\subset\mathcal{M}\cap\partial\Omega$ by
\begin{align*}
  \Gamma_w&=\{x\in \partial\Omega\cap\mathcal{M}\mid w(x)\leq g(x)\},\\
  \Gamma_v&=\{x\in \partial\Omega\cap\mathcal{M}\mid v(x)\geq g(x)\}.
\end{align*}
\textbf{Case 1.} The case that $\Gamma_w\cup\Gamma_v =\emptyset$.
We consider the following auxiliary function:
\begin{align*}
  \Phi_\varepsilon (x,y):=w(x)-v(y)-\phi_\varepsilon (x,y),
  \quad \phi_\varepsilon (x,y) := \left|\frac{x-y}{\varepsilon}\right|^4,
  \quad \varepsilon >0,\ x,y\in \overline{\Omega}.
\end{align*}
Let $(x_\varepsilon,y_\varepsilon)$ be a point where 
$\Phi_\varepsilon$ attains its global maximum on $\overline{\Omega}\times\overline{\Omega}$.
Then, by the inequality
$\Phi_\varepsilon (x_\varepsilon,y_\varepsilon)\geq \Phi_\varepsilon (x_0,x_0)=M$
and boundedness of $w$ and $v$, we have
\begin{align}\label{eq:CP-pf1}
  |x_\varepsilon -y_\varepsilon| \leq C\varepsilon.
\end{align}
Furthermore, since (\ref{eq:CP-pf1}) and $|x_\varepsilon|,|y_\varepsilon|\leq R$ hold for some $R>0$,
if necessary, by taking a subsequence, 
we can assume that $x_\varepsilon,y_\varepsilon\to \overline{x}\in\overline{\Omega}$ as $\varepsilon\to 0$.
Then, the upper semicontinuity of $w-v$ implies
\begin{align*}
  \limsup_{\varepsilon\to 0} \Phi_\varepsilon (x_\varepsilon,y_\varepsilon)
  \leq \limsup_{\varepsilon\to 0} \left(w(x_\varepsilon)-v(y_\varepsilon )\right)
  \leq w(\overline{x})-v(\overline{x})\leq M.
\end{align*}
On the other hand, since it follows that 
\begin{align*}
  \liminf_{\varepsilon \to 0} \Phi_\varepsilon (x_\varepsilon,y_\varepsilon)
  \geq \liminf_{\varepsilon\to 0} \Phi_\varepsilon (x_0,x_0)
  =w(x_0)-v(x_0)=M,
\end{align*}
we obtain $\lim_{\varepsilon \to 0 } \Phi_\varepsilon (x_\varepsilon,y_\varepsilon)=M$.
From this equality, we can deduce that
\begin{align}\label{eq:CP-pf2}
  |x_\varepsilon -y_\varepsilon |=o(\varepsilon),\quad
  w(x_\varepsilon)-v(y_\varepsilon)\to w(\overline{x})-v(\overline{x})=M
  \quad \text{as } \varepsilon \to 0.
\end{align}
Indeed, 
by choosing a subsequence $\{\varepsilon_k\}$ such that
\begin{align*}
  \lim_{k\to\infty}\phi_{\varepsilon_k}(x_{\varepsilon_k},y_{\varepsilon_k})
  =\limsup_{\varepsilon \to 0} \phi_\varepsilon (x_\varepsilon,y_\varepsilon),
\end{align*}
we have
\begin{align*}
  M&=\lim_{\varepsilon \to 0} \Phi_\varepsilon (x_\varepsilon,y_\varepsilon)
  =\lim_{k\to\infty} \Phi_{\varepsilon_k} (x_{\varepsilon_k},y_{\varepsilon_k})\\
  &\leq \limsup_{k\to\infty} \left(w(x_{\varepsilon_k})-v(y_{\varepsilon_k})\right)
  -\lim_{k\to\infty} \phi_{\varepsilon_k}(x_{\varepsilon_k},y_{\varepsilon_k})\\
  &\leq w(\overline{x})-v(\overline{x})
  -\limsup_{\varepsilon\to 0}\phi_\varepsilon (x_\varepsilon,y_\varepsilon)
  \leq M-\limsup_{\varepsilon\to 0}\phi_\varepsilon (x_\varepsilon,y_\varepsilon),
\end{align*}
and therefore,
$\limsup_{\varepsilon\to 0} \phi_\varepsilon (x_\varepsilon,y_\varepsilon)\leq 0$.
Combining this with $\lim_{\varepsilon\to 0 }\Phi_\varepsilon (x_\varepsilon,y_\varepsilon)=M$,
we can get (\ref{eq:CP-pf2}).

The assumption $\Gamma_w\cup\Gamma_v=\emptyset$ implies that
\begin{align*}
  \text{either}\quad 
  \overline{x}\in\Omega \quad \text{or}\quad 
  \overline{x}\in\partial\Omega,\ w(\overline{x})>g(\overline{x})>v(\overline{x}).
\end{align*}
Furthermore, 
if $\overline{x}\in\partial\Omega$, then
\begin{align}\label{eq:CP-pf10}
  \left\{
  \begin{aligned}
    &x_\varepsilon \in \partial \Omega \ \Rightarrow\ w(x_\varepsilon)>g(x_\varepsilon),\\
    &y_\varepsilon \in \partial \Omega \ \Rightarrow\ v(y_\varepsilon)<g(y_\varepsilon)
  \end{aligned}
  \right.
\end{align}
holds for sufficiently small $\varepsilon>0$ by (\ref{eq:CP-pf2}).
By (\ref{eq:CP-pf10}) and Dirichlet boundary conditions of $w$ and $v$,
it follows that 
\begin{align}\label{eq:CP-pf3}
  w(x_\varepsilon)+F_*(Dw(x_\varepsilon),D^2w(x_\varepsilon))\leq 1,\quad 
  v(y_\varepsilon)+F^*(Dv(y_\varepsilon),D^2v(y_\varepsilon))\geq 1
\end{align}
in the sense of viscosity solutions for any sufficiently small $\varepsilon >0$.

Since the function $(x,y)\mapsto \Phi_\varepsilon (x,y)=w(x)-v(y)-\phi_\varepsilon(x,y)$ 
attains its maximum on $\overline{\Omega}\times\overline{\Omega}$ at 
$(x,y)=(x_\varepsilon,y_\varepsilon)$,
by Ishii's lemma in \cite{MR1118699},
for any $\alpha> 0$,
there exist $(p,X)\in \overline{J}_{\overline{\Omega}}^{2,+}w(x_\varepsilon)$ and 
$(q,Y)\in \overline{J}_{\overline{\Omega}}^{2,-}v(y_\varepsilon)$ such that
\begin{align}\label{eq:CP-ishii-lem}
  p=D_x\phi_\varepsilon(x_\varepsilon,y_\varepsilon),\quad 
  &q=-D_y \phi_\varepsilon (x_\varepsilon,y_\varepsilon),\notag\\
  -\left(\frac{1}{\alpha}+\lVert D^2\phi_\varepsilon(x_\varepsilon,y_\varepsilon)\rVert\right) I_{2n}
  \leq 
  \begin{pmatrix}
    X & O\\
    O & -Y
  \end{pmatrix}
  &\leq  (I_{2n} +\alpha D^2\phi_\varepsilon (x_\varepsilon,y_\varepsilon))D^2\phi_\varepsilon (x_\varepsilon,y_\varepsilon).
\end{align}
We choosing $\alpha=\varepsilon^4$ for each $\varepsilon>0$,
$(p,X)$ and $(q,Y)$ with (\ref{eq:CP-ishii-lem}) satisfy
\begin{align}\label{eq:CP-pf4}
  |p-q|=0\leq \varepsilon (|p|\wedge|q|),\quad 
  \begin{pmatrix}
    X & O \\
    O & -Y
  \end{pmatrix}
  \leq \frac{o(1)}{\varepsilon^2}
  \begin{pmatrix}
    I_n & -I_n \\
    -I_n & I_n
  \end{pmatrix}
  + o(1)I_{2n} \quad \text{as } \varepsilon \to 0
\end{align}
since 
\begin{align*}
  &p=q=\frac{4}{\varepsilon}\left|\widetilde{p}\right|^2
  \widetilde{p},\\
  &D^2\phi_\varepsilon(x,y)
  =\frac{4}{\varepsilon^2}|\widetilde{p}|^2
  \begin{pmatrix}
    I_n & -I_n \\
    -I_n & I_n
  \end{pmatrix}
  +\frac{8}{\varepsilon^2}
  \begin{pmatrix}
    \widetilde{p}\otimes \widetilde{p} & -\widetilde{p}\otimes \widetilde{p}\\
    -\widetilde{p}\otimes \widetilde{p} & \widetilde{p}\otimes \widetilde{p}
  \end{pmatrix}
  \leq \frac{C}{\varepsilon^2}|\widetilde{p}|^2
  \begin{pmatrix}
    I_n & -I_n \\
    -I_n & I_n
  \end{pmatrix}
  ,\\
  &\text{where } \widetilde{p}:=\frac{x_\varepsilon -y_\varepsilon}{\varepsilon},
\end{align*}
and $\widetilde{p}=o(1)$.

Since $F(p,X)$ has singularity at $p=0$,
we need to consider two cases in order to substitute (\ref{eq:CP-pf4})
into (\ref{eq:CP-pf3}).
First, assume that $x_\varepsilon \neq y_\varepsilon$ for any small $\varepsilon>0$.
Then, (\ref{eq:CP-pf3}), (\ref{eq:CP-pf4}) and Proposition \ref{lem-CP} implies
\begin{align*}
  0< M \leq \Phi_\varepsilon (x_\varepsilon,y_\varepsilon)
  \leq w(x_\varepsilon)-v(y_\varepsilon)
  \leq F(q,Y)-F(p,X)\leq C\varepsilon (|p|\wedge|q|)+o(1)
  \quad \text{as }\varepsilon\to 0.
\end{align*}
We have a contradiction by letting $\varepsilon \to 0$.

Otherwise, 
if there is a subsequence $\{\varepsilon_k\}$ such that 
$\varepsilon_k\to 0$ and $x_{\varepsilon_k}=y_{\varepsilon_k}$, 
then it follows that  $\widetilde{p}=0$ and $X,-Y\leq O$ 
for all such $\varepsilon_k$. 
Therefore, 
\begin{align*}
  0< M \leq \Phi_{\varepsilon_k} (x_{\varepsilon_k},y_{\varepsilon_k})
  \leq w(x_{\varepsilon_k})-v(y_{\varepsilon_k})
  \leq F^*(0,Y)-F_*(0,X)\leq F^*(0,O)-F_*(0,O)=0.
\end{align*}
It is a contradiction.\\
\textbf{Case 2.} The case that $\Gamma_w\cup\Gamma_v\neq \emptyset$.
By \cite[Lemma 3.1]{MR1702961}, 
we can see that 
$\Gamma_w$ and $\Gamma_v$ are possibly empty closed sets and disjoint.
Then, there exists a smooth function $\zeta \in C^\infty(\R^n)$ 
such that $\supp\zeta \subset \mathcal{W}$, 
$\zeta\equiv 1$ on a neighborhood of $\Gamma_w$ and
$\zeta \equiv -1$ on a neighborhood of $\Gamma_v$.
By using the mollifier, 
we can assume that the function 
$\mathbf{n}$ in assumption (A\ref{assum:interior-cone}) is smooth.
Thus, the function $\chi:\R^n\to\R^n$ defined by
\begin{align*}
  \chi (x) :=\left\{\ 
  \begin{aligned}
    &\zeta(x)\mathbf{n}(x) &&\text{if }x\in \mathcal{W},\\
    &0 &&\text{if } x\not\in\mathcal{W}
  \end{aligned}
  \right.
\end{align*}
is smooth on $\R^n$.
In this case, we consider the following auxiliary function:
\begin{align*}
  \Phi_\varepsilon (x,y)= w(x)-v(y)-\phi_\varepsilon(x,y),\quad
  \phi_\varepsilon (x,y)=\left|\frac{x-y}{\varepsilon} + \chi \left(\frac{x+y}{2}\right)\right|^4,\quad
  \varepsilon >0,\ x,y\in\overline{\Omega}.
\end{align*}
Let $(x_\varepsilon,y_\varepsilon)$ be a point where 
$\Phi_\varepsilon$ attains its global maximum on $\overline{\Omega}\times\overline{\Omega}$.
Then, by condition (\ref{assum:w-and-v-1}) and the same argument
of the proof of \cite[Theorem 3.1]{MR1702961}, the following property holds:
if necessary, by taking a subsequence, 
$x_\varepsilon,y_\varepsilon\to \overline{x}\in\overline{\Omega}$ as $\varepsilon\to 0$
and $\Phi_\varepsilon(x_\varepsilon,y_\varepsilon)\to M$ as $\varepsilon\to 0$.
Furthermore, 
\begin{align}\label{eq:CP-pf5}
  \phi_\varepsilon(x_\varepsilon,y_\varepsilon)=o(1),\quad 
  w(x_\varepsilon)-v(y_\varepsilon)\to w(\overline{x})-v(\overline{x})=M \quad 
  \text{as } \varepsilon \to 0.
\end{align}

Then, by (\ref{eq:CP-pf5}), we obtain:
for sufficiently small $\varepsilon >0$,
\begin{align}\label{eq:CP-pf6}
  \left\{
  \begin{aligned}
    &x_\varepsilon,y_\varepsilon\in \Omega \text{ if } \overline{x}\not\in \Gamma_w\cup\Gamma_v;\\
    &x_\varepsilon \in \Omega\text{ and either }y_\varepsilon\in \Omega
    \text{ or } y_\varepsilon\in\partial\Omega, v(y_\varepsilon)<g(y_\varepsilon)
    \text{ if }\overline{x}\in \Gamma_w;\\
    &y_\varepsilon\in\Omega \text{ and either }
    x_\varepsilon \in \Omega \text{ or } x_\varepsilon\in\partial\Omega, w(x_\varepsilon)>g(x_\varepsilon)
    \text{ if }\overline{x}\in \Gamma_v.
  \end{aligned}
  \right.
\end{align}
Therefore, we can apply (\ref{eq:CP-pf3}) for each $\varepsilon>0$ in this case as well.

Similarly to Case 1, for each $\varepsilon>0$ and $\alpha>0$,
there exist $(p,X)\in \overline{J}_{\overline{\Omega}}^{2,+}w(x_\varepsilon)$ and
$(q,Y)\in\overline{J}_{\overline{\Omega}}^{2,-}v(y_\varepsilon)$
satisfying condition (\ref{eq:CP-ishii-lem}).
For the auxiliary function considered in this case, 
since we can compute
\begin{align}\label{eq:CP-pf9}
  &p=\frac{4}{\varepsilon}|\widetilde{p}|^2
  \left(I_n+\frac{\varepsilon}{2}{}^t\!D\chi\left(\frac{x_\varepsilon+y_\varepsilon}{2}\right)\right)
  \widetilde{p},\quad 
  q=\frac{4}{\varepsilon}|\widetilde{p}|^2
  \left(I_n-\frac{\varepsilon}{2}{}^t\!D\chi\left(\frac{x_\varepsilon+y_\varepsilon}{2}\right)\right)
  \widetilde{p},\\
  &\text{where } \widetilde{p}:=\frac{x_\varepsilon-y_\varepsilon}{\varepsilon}+\chi\left(\frac{x+y}{2}\right),\notag
\end{align}
by teh similar calculations to the proof of \cite[Theorem 3.1]{MR1702961}
and \cite[Appendix A]{MR3365833}, 
for an appropriate choice of $\alpha>0$ for each $\varepsilon>0$,
it follows that
\begin{align}\label{eq:CP-pf7}
  \begin{pmatrix}
    X & O\\
    O & -Y
  \end{pmatrix}
  \leq \frac{o(1)}{\varepsilon^2}
  \begin{pmatrix}
    I_n & -I_n \\
    -I_n & I_n
  \end{pmatrix}
  +o(1) I_{2n}\quad\text{as } \varepsilon \to 0
\end{align}
and
\begin{align}\label{eq:CP-pf8}
  |p-q|=4|\widetilde{p}|^2 \;{}^t\! D\chi\left(\frac{x_\varepsilon+y_\varepsilon}{2}\right)
  \widetilde{p}\leq C\varepsilon (|p|\wedge |q|)
\end{align}
for small $\varepsilon>0$.

Again, we consider two cases regarding the singularity of $F$.
If $p\neq 0$ and $q\neq 0$ hold for every small $\varepsilon >0$,
then (\ref{eq:CP-pf5}), (\ref{eq:CP-pf6}), (\ref{eq:CP-pf7}), (\ref{eq:CP-pf8})
and Proposition \ref{lem-CP} implies
\begin{align*}
  0<\frac{M}{2}\leq \Phi_\varepsilon(x_\varepsilon,y_\varepsilon)
  \leq w(x_\varepsilon)-v(y_\varepsilon)
  \leq F(q,Y)-F(p,X)\leq C\varepsilon (|p|\wedge|q|)+o(1) \quad \text{as }\varepsilon \to 0.
\end{align*}
Letting $\varepsilon\to 0$, we get a contradiction.

If there is a subsequence $\{\varepsilon_k\}$ such that
$\varepsilon_k\to 0$, $p=0$,
then it is necessary that $\widetilde{p}=0$
holds for small $\varepsilon_k$ by 
(\ref{eq:CP-pf9}).
This implies that $q=0$ and then,
we have $p=q=0$, $X,-Y\leq O$ and 
\begin{align*}
  0<\frac{M}{2}\leq \Phi_{\varepsilon_k}(x_{\varepsilon_k},y_{\varepsilon_k})
  \leq w(x_{\varepsilon_k})-v(y_{\varepsilon_k})
  \leq F^*(0,Y)-F_*(0,X)\leq F^*(0,O)-F_*(0,O)=0.
\end{align*}
It is a contradiction.
For the case that $q=0$ for a subsequence
$\{\varepsilon_k\}$, we can get a contradiction by the same argument.
\end{proof}

\section{Game interpretation}
In this section, 
we consider the deterministic two-person games for our equations, 
which is an extension of the game proposed in \cite{MR2200259}.

\subsection{The game setting}
We consider the game played by two players named Player I and 
Player II. 
The set of direction choice $\mathcal{D}$ for Player I 
and the set of sign choice $\mathcal{S}$ for Player II 
are defined by 
\begin{align}\label{dfn:strategy-sets}
  \begin{aligned}
  \mathcal{D} &= \mathcal{D}_1 \times \mathcal{D}_2,\\
  \mathcal{D}_1&=\{ \mathbf{v}=(v^1,v^2)\in \R^n\times\R^n \mid |v^1|=|v^2|=1 \},\\ 
  \mathcal{D}_2&=\{ \mathbf{w}=(w^1,\ldots, w^m)\in \R^{m\times m} \mid \{w^i\} \text{: orthonormal basis of }\R^m\},\\
  \mathcal{S}&=\{ \mathbf{b}=(b^1,\ldots,b^m)\in \{\pm 1\}^{m} \}.
  \end{aligned}
\end{align}
Let $D_0$ be a given open set, 
$G$ be a given continuous function defined on a neighborhood of $\partial D_0$,
$x\in\R^n \setminus \overline{D}_0$ be an initial state 
and $\varepsilon>0$ be a step size for space. 
In the begining, there is a marker at point $x$. 
On each turn, Player I and Player II choose 
$(\mathbf{v},\mathbf{w})\in\mathcal{D}$ 
and $\mathbf{b}\in\mathcal{S}$, respectively, and depending on their choices, 
the marker moves to the next position according to the following rules: 
first, set $y_0=x$. 
At each $j$-th turn, 
if the marker was at point $y_{j-1}\in\R^n\setminus \overline{D}_0$ at the end of 
$(j-1)$-th turn, 
\begin{enumerate}[(1)]
  \item Player I chooses the set of directions 
  $(\mathbf{v}_j,\mathbf{w}_j)\in\mathcal{D}$.
  \item Player II chooses the set of $m$ signs 
  $\mathbf{b}_j\in\mathcal{S}$ 
  after seeing Player I's choice.
  \item The marker moves to the next position $y_j$ determined by
  \begin{align*}
    y_j:=y_{j-1}+\delta^\varepsilon(\mathbf{v}_j,\mathbf{w}_j,\mathbf{b}_j),
  \end{align*}
  where
  \begin{align*}
    \delta^\varepsilon(\mathbf{v},\mathbf{w},\mathbf{b})
    :=\varepsilon \sqrt{2}\sum_{i=1}^m b^i\sigma(v^1)w^i + \varepsilon^2 c(v^1)v^2.
  \end{align*}
\end{enumerate}
Players repeat these steps and the game ends 
when the marker reaches the set $\overline{D}_0$.
Suppose the number of steps taken for the marker to reach $\overline{D}_0$ is $N\in\N$,
then we define the Player I's payoff in this game as $\varepsilon^2 N + G(y_N)$.
If the game never ends, we define the payoff is $+\infty$.
Player I's goal of the game is to hit the marker to $\overline{D}_0$ 
and to minimize the payoff, while player II's goal is to maximize it.
Player I makes a rational choice of $(\mathbf{v}_j,\mathbf{w}_j)$ each turn 
so as to minimize the payoff until the marker reaches $\overline{D}_0$, 
and Player II makes a rational choice of $\mathbf{b}_j$ 
each turn to maximize it. 
Let $N^\varepsilon(x)$ and $y^\varepsilon (x)=y_{N^\varepsilon(x)}$ be the number of steps for the marker to reach $\overline{D}_0$ 
and the point where the marker hits $\overline{D}_0$, respectively,
when each player continues to make rational choices.
In the view of the above, 
we define the value of the game $U^\varepsilon(x)$, 
although it may not be well-defined, as follows
\begin{align*}
  U^\varepsilon (x)=\varepsilon^2 N^\varepsilon (x) + G(y^\varepsilon (x)).
\end{align*}
(If the game never ends, then we define $U^\varepsilon (x)=+\infty$.)
If $U^\varepsilon$ exists, 
then $U^\varepsilon$ satisfies the following dynamic programming principle
\begin{align}\label{DPP-U}
  U^\varepsilon (x)
  =\inf_{(\mathbf{v},\mathbf{w})\in \mathcal{D}}\sup_{\mathbf{b}\in\mathcal{S}}
  \left\{
  \begin{aligned}
    &\varepsilon^2 +U^\varepsilon (x+\delta^\varepsilon(\mathbf{v},\mathbf{w},\mathbf{b}))
    &&\text{if}\ x+\delta^\varepsilon(\mathbf{v},\mathbf{w},\mathbf{b})\in \R^n\setminus \overline{D}_0, \\
    &\varepsilon^2 +G(x+\delta^\varepsilon (\mathbf{v},\mathbf{w},\mathbf{b}))
    &&\text{if}\ x+\delta^\varepsilon(\mathbf{v},\mathbf{w},\mathbf{b})\in \overline{D}_0.
  \end{aligned}
  \right.
\end{align}

We can define the value of the game $U^\varepsilon$ in a formal way in fact
(we mention this in subsection 3.2 and 3.3)
and we consider two functions
\begin{align}\label{eq:def-U}
  \overline{U}(x):= \limsup_{\substack{\R^n\setminus \overline{D}_0\ni y\to x \\ \varepsilon \to 0}}
  U^\varepsilon(y),\quad 
  \underline{U}(x):=\liminf_{\substack{\R^n\setminus \overline{D}_0\ni y\to x \\ \varepsilon \to 0}}
  U^\varepsilon(y),\quad 
  x\in \R^n\setminus D_0.
\end{align}
Then, we have the following result.
\begin{thm}\label{thm:U-conv}
Assume $(\mathrm{A}\ref{assum:homogenity})$-$(\mathrm{A}\ref{assum:rank})$. 
Then $\overline{U}$ $($resp., $\underline{U}$$)$ satisfies 
\begin{align*}
\left\{
  \begin{aligned}
  F(D\overline{U},D^2\overline{U})&\leq 1
  &&\text{in } \{\overline{U}<\infty\}\cap (\R^n\setminus \overline{D}_0),\\
  \overline{U}&\leq G
  &&\text{on } \{\overline{U}<\infty\}\cap \partial D_0
  \end{aligned}
\right.
\end{align*}
in the sense of viscosity solutions $($resp.,
\begin{align*}
\left\{
  \begin{aligned}
  F(D\underline{U},D^2\underline{U})&\geq 1
  &&\text{in } \{\overline{U}<\infty\}\cap (\R^n\setminus \overline{D}_0),\\
  \underline{U}&\geq G
  &&\text{on } \{\overline{U}<\infty\}\cap \partial D_0
  \end{aligned}
\right.
\end{align*}
in the sense of viscosity solutions $)$.
\end{thm}

This theorem is 
obtained by combining Theorem \ref{thm:game-conv} in subsection 3.3 
and Corollary \ref{cor:trans}.

\subsection{Change of variables for the value functions}
We need change of variables for the value function $U^\varepsilon$ 
since it can be $+\infty$ and 
we would like to treat the value function with real-valued
defined on the entire set $\R^n\setminus \overline{D}_0$. 
Consider the function $\psi$ defined by (\ref{dfn:psi}),
and extend it to $\psi :\R\cup\{+\infty\}\to (-\infty,1]$ 
as
\begin{align*}
  \psi(r) = \left\{
  \begin{aligned}
    & 1-e^{-r} &&\text{if } r<+\infty,\\
    & 1 && \text{if } r=+\infty
  \end{aligned}
  \right.
\end{align*}
and let  
\begin{align}\label{trans-value}
  u^\varepsilon (x) = \psi(U^\varepsilon (x)).
\end{align}
Substituting (\ref{trans-value}) into 
the dynamic programming principle (\ref{DPP-U}) and calculating that, 
we have the dynamic programming principle for $u^\varepsilon$ as follows:
\begin{align}\label{discounted-DPP}
  u^\varepsilon (x)
  =\min_{(\mathbf{v},\mathbf{w})\in \mathcal{D}}\max_{\mathbf{b}\in\mathcal{S}}
  \left\{
  \begin{aligned}
    &1-e^{-\varepsilon^2}+e^{-\varepsilon^2}
    u^\varepsilon (x+\delta^\varepsilon(\mathbf{v},\mathbf{w},\mathbf{b}))
    &&\text{if}\ x+\delta^\varepsilon(\mathbf{v},\mathbf{w},\mathbf{b})\in \R^n\setminus \overline{D}_0,\\
    &1-e^{-\varepsilon^2}+e^{-\varepsilon^2}\psi (G(x+\delta^\varepsilon(\mathbf{v},\mathbf{w},\mathbf{b})))
    &&\text{if}\ x+\delta^\varepsilon(\mathbf{v},\mathbf{w},\mathbf{b})\in \overline{D}_0.
  \end{aligned}
  \right.
\end{align}
This is a similar representation 
to the game with discounting 
for stationary boundary value problems studied in \cite{MR2681474}.

Similarly to (\ref{eq:def-U}), considering 
\begin{align*}
  \overline{u}(x)
  :=\limsup_{\substack{\R^n\setminus \overline{D}_0\ni y\to x\\ \varepsilon\to 0 }}
  u^\varepsilon (y)\quad \text{and} \quad
  \underline{u}(x)
  :=\liminf_{\substack{\R^n\setminus \overline{D}_0\ni y\to x \\ \varepsilon\to 0}}
  u^\varepsilon (y),
\end{align*}
we have the result below.
\begin{thm}\label{thm:u-conv}
The function $\overline{u}$ $($resp., $\underline{u}$ $)$ satisfies
\begin{align*}
  \left\{
    \begin{aligned}
    \overline{u} + F(D\overline{u},D^2\overline{u})&\leq 1
    &&\text{in } \R^n\setminus \overline{D}_0,\\
    \overline{u}&\leq \psi \circ G
    &&\text{on } \partial D_0
    \end{aligned}
  \right.
  \end{align*}
  in the sense of viscosity solutions $($resp.,
\begin{align*}
\left\{
  \begin{aligned}
  \underline{u} + F(D\underline{u},D^2\underline{u})&\geq 1
  &&\text{in } \R^n\setminus \overline{D}_0,\\
  \underline{u}&\geq \psi\circ G
  &&\text{on } \partial D_0
  \end{aligned}
\right.
\end{align*}
in the sense of viscosity solutions $)$.
\end{thm}

Theorem \ref{thm:u-conv} is a straight forward result of
Theorem \ref{thm:game-conv} below and the proof is given 
at the end of Section 3.

\begin{prop}\label{prop:fixed-point}
Assume that 
$(\mathrm{A}\ref{assum:homogenity})$-$(\mathrm{A}\ref{assum:rank})$
hold.
Let $\Omega\subset \R^n$ be an open set, 
$\mathcal{B}(\Omega)$ be a complete normed space 
of bounded real valued functions on $\Omega$ 
and $g$ be a bounded contiuous function defined on a neighborhood 
of $\partial \Omega$.
Then, for each small $\varepsilon>0$, the operator 
$R^\varepsilon: \mathcal{B}(\Omega)\to \mathcal{B}(\Omega)$ 
defined by 
\begin{align}\label{DPP-u-2}
  R^\varepsilon [\phi](x)
  =\inf_{(\mathbf{v},\mathbf{w})\in \mathcal{D}}\max_{\mathbf{b}\in\mathcal{S}}
  \left\{
    \begin{aligned}
      &1-e^{-\varepsilon^2}+e^{-\varepsilon^2}
      \phi (x+\delta^\varepsilon(\mathbf{v},\mathbf{w},\mathbf{b}))
      &&\text{if}\ x+\delta^\varepsilon(\mathbf{v},\mathbf{w},\mathbf{b}) \in \Omega,\\
      &1-e^{-\varepsilon^2}+e^{-\varepsilon^2}
      g (x+\delta^\varepsilon(\mathbf{v},\mathbf{w},\mathbf{b}))
      &&\text{if}\ x+\delta^\varepsilon(\mathbf{v},\mathbf{w},\mathbf{b})\not \in \Omega
    \end{aligned}
    \right.
\end{align}
is a contraction mapping on $\mathcal{B}(\Omega)$. 
In particular, there is a unique fixed point $\phi^\varepsilon\in \mathcal{B}(\Omega)$, 
which satisfies $R^\varepsilon[\phi^\varepsilon]=\phi^\varepsilon$ in $\Omega$.
\end{prop}

\begin{proof}
Take $\phi_1,\phi_2\in \mathcal{B}(\Omega)$ arbitrarily. 
For every $\varepsilon>0$ and $x\in\Omega$, it follows that 
\begin{align*}
  e^{-\varepsilon^2}\phi_1(x+\delta^\varepsilon(\mathbf{v},\mathbf{w},\mathbf{b}))
  \leq e^{-\varepsilon^2}\phi_2(x+\delta^\varepsilon(\mathbf{v},\mathbf{w},\mathbf{b}))
  +e^{-\varepsilon^2}\sup_{\Omega}|\phi_1-\phi_2|
\end{align*}
whenever $x+\delta^\varepsilon(\mathbf{v},\mathbf{w},\mathbf{b})\in\Omega$.
This implies that 
\begin{align*}
  R^\varepsilon[\phi_1](x)\leq R^\varepsilon[\phi_2](x)+e^{-\varepsilon^2}\sup_{\Omega}|\phi_1-\phi_2|
\end{align*}
for all $x\in\Omega$. By symmetricity, we have 
\begin{align*}
  |R^\varepsilon [\phi_1](x)-R^\varepsilon [\phi_2](x)|\leq e^{-\varepsilon^2}\sup_{\Omega}|\phi_1-\phi_2|
\end{align*}
for all $x\in\Omega$. Therefore, $R^\varepsilon$ is a contraction mapping,
which implies that 
there is a unique fixed point by the Banach fixed point theorem.
\end{proof}

\begin{prop}\label{consistency-with-games}
  Assume that 
  $(\mathrm{A}\ref{assum:homogenity})$-$(\mathrm{A}\ref{assum:rank})$
  hold.
  Let $\Omega\subset \R^n$ be an open set and 
  $g$ be a bounded continuous function defined on a neighborhood of $\partial\Omega$.
  For each small $\varepsilon>0$, 
  let $u^\varepsilon\in\mathcal{B}(\Omega) $ 
  be the fixed point of $R^\varepsilon[\;\cdot\;]$ defined in 
  Proposition $\ref{prop:fixed-point}$, 
  i.e., $u^\varepsilon\in \mathcal{B}(\Omega)$ is a function uniquely determined by
  \begin{align*}
    u^\varepsilon (x)
  = \inf_{(\mathbf{v},\mathbf{w})\in \mathcal{D}}\max_{\mathbf{b}\in\mathcal{S}}
  \left\{
  \begin{aligned}
    &1-e^{-\varepsilon^2}+e^{-\varepsilon^2}
      u^\varepsilon (x+\delta^\varepsilon(\mathbf{v},\mathbf{w},\mathbf{b}))
      &&\text{if}\ x+\delta^\varepsilon(\mathbf{v},\mathbf{w},\mathbf{b}) \in \Omega,\\
      &1-e^{-\varepsilon^2}+e^{-\varepsilon^2}
      g (x+\delta^\varepsilon(\mathbf{v},\mathbf{w},\mathbf{b}))
      &&\text{if}\ x+\delta^\varepsilon(\mathbf{v},\mathbf{w},\mathbf{b})\not \in \Omega.
  \end{aligned}
  \right.
  \end{align*}
  Then, for any $x\in\Omega$ and $\alpha>0$, 
  there exist sequences 
  $\{(\mathbf{v}_j,\mathbf{w}_j,\mathbf{b}_j)\}_{j=1}^N\ (N\in\N\cup\{\infty\})$ and 
  $\{y_j\}_{j=0}^N$ such that 
  $y_0=x$, 
  $y_j = y_{j-1}+\delta^\varepsilon(\mathbf{v}_j,\mathbf{w}_j,\mathbf{b}_j)$
  and either $N<\infty$, $y_N\not\in\Omega$ and  
  \begin{align*}
    1-e^{-N\varepsilon^2}+e^{-N\varepsilon^2}g(y_N)-\alpha < u^\varepsilon (x)
    \leq 1-e^{-N\varepsilon^2}+e^{-N\varepsilon^2}g(y_N),
  \end{align*}
  or else, $N=\infty$ and $u^\varepsilon (x)=1$.
\end{prop}

\begin{proof}
For the notational simplicity, define $\widehat{u}^\varepsilon$ on an open neighborhood of 
$\overline{\Omega}$ by 
\begin{align*}
  \widehat{u}^\varepsilon (x)
  =\left\{
  \begin{aligned}
    & u^\varepsilon (x) &&\text{if } x\in\Omega,\\
    & g(x) && \text{if } x\not\in \Omega.
  \end{aligned}
  \right.
\end{align*}
Under this notation, the dynamic programming principle for $u^\varepsilon$
can be written as
\begin{align*}
  u^\varepsilon(x)=
  \inf_{(\mathbf{v},\mathbf{w})\in \mathcal{D}}\max_{\mathbf{b}\in\mathcal{S}}
  \left\{ 1-e^{-\varepsilon^2}+
  e^{-\varepsilon^2}\widehat{u}^\varepsilon (x+\delta^\varepsilon(\mathbf{v},\mathbf{w},\mathbf{b}))
  \right\}.
\end{align*}

Fix $x\in\Omega$ and $\alpha >0$.
As the first step, choose 
$(\mathbf{v}_1,\mathbf{w}_1)\in\mathcal{D}$ such that 
\begin{align*}
  u^\varepsilon (x)+\frac{\alpha}{2}
  > \max_{\mathbf{b}\in\mathcal{S}} \left\{ 1-e^{-\varepsilon^2}+
  e^{-\varepsilon^2}
  \widehat{u}^\varepsilon(x+\delta^\varepsilon (\mathbf{v}_1,\mathbf{w}_1,\mathbf{b}))
  \right\},
\end{align*}
and then, choose $\mathbf{b}_1\in\mathcal{S}$ such that 
\begin{align*}
  \widehat{u}^\varepsilon(x+\delta^\varepsilon(\mathbf{v}_1,\mathbf{w}_1,\mathbf{b}_1))
  =\max_{\mathbf{b}\in\mathcal{S}}\widehat{u}^\varepsilon(x+\delta^\varepsilon (\mathbf{v}_1,\mathbf{w}_1,\mathbf{b})).
\end{align*}
Set $y_1:=x+\delta^\varepsilon (\mathbf{v}_1,\mathbf{w}_1,\mathbf{b}_1)$.
At this time, if $y_1\not \in \Omega$, set $N=1$. 
If $y_1\in \Omega$, go to the second step: 
choose $(\mathbf{v}_2,\mathbf{w}_2)\in\mathcal{D}$ such that 
\begin{align*}
  u^\varepsilon (y_1)+\frac{\alpha}{4}
  > \max_{\mathbf{b}\in\mathcal{S}} \left\{ 1-e^{-\varepsilon^2}+
  e^{-\varepsilon^2}\widehat{u}^\varepsilon(y_1+\delta^\varepsilon (\mathbf{v}_2,\mathbf{w}_2,\mathbf{b}))
  \right\}
\end{align*}
and $\mathbf{b}_2\in\mathcal{S}$ such that 
\begin{align*}
  \widehat{u}^\varepsilon(y_1+\delta^\varepsilon(\mathbf{v}_2,\mathbf{w}_2,\mathbf{b}_2))
  =\max_{\mathbf{b}\in\mathcal{S}}
  \widehat{u}^\varepsilon(y_1+\delta^\varepsilon (\mathbf{v}_2,\mathbf{w}_2,\mathbf{b})).
\end{align*}
Set
$y_2=y_1 + \delta^\varepsilon (\mathbf{v}_2,\mathbf{w}_2,\mathbf{b}_2)$.
Similarly, if $y_2\not\in\Omega$ at this time, set $N=2$. 
And if $y_2\in\Omega$, go to the next step.
For each $j=1,2,\ldots$, we iteratively choose $(\mathbf{v}_j,\mathbf{w}_j)$ with 
\begin{align*}
  u^\varepsilon (y_{j-1})+\frac{\alpha}{2^{j}}
  > \max_{\mathbf{b}\in\mathcal{S}}\left\{ 1-e^{-\varepsilon^2}+
  e^{-\varepsilon^2}\widehat{u}^\varepsilon (y_{j-1}+\delta^\varepsilon (\mathbf{v}_j,\mathbf{w}_j,\mathbf{b}))
  \right\}
\end{align*}
and pick $\mathbf{b}_j$ with 
\begin{align*}
  \widehat{u}^\varepsilon(y_{j-1}+\delta^\varepsilon(\mathbf{v}_j,\mathbf{w}_j,\mathbf{b}_j))
  =\max_{\mathbf{b}\in\mathcal{S}}\widehat{u}^\varepsilon(y_{j-1}+\delta^\varepsilon (\mathbf{v}_j,\mathbf{w}_j,\mathbf{b}))
\end{align*}
until $y_j:=y_{j-1}+\delta^\varepsilon (\mathbf{v}_j,\mathbf{w}_j,\mathbf{b}_j)\not\in\Omega$ occurs. 
We let $N=j$ when $y_j\not\in \Omega$ for the first time.
Then, by the choices of $(\mathbf{v}_j,\mathbf{w}_j,\mathbf{b}_j)$, 
it follows that 
\begin{align}\label{eq:consistency-of-games-pf1}
  u^\varepsilon (x)+\alpha
  &=u^\varepsilon (x)+\alpha\sum_{j=0}^\infty \frac{1}{2^{j+1}}\\
  &>u^\varepsilon (x)+\sum_{j=0}^\infty \frac{\alpha e^{-j\varepsilon^2}}{2^{j+1}}
  =u^\varepsilon (x)+\frac{\alpha}{2} + \frac{\alpha e^{-\varepsilon^2}}{4}
  +\sum_{j=2}^\infty \frac{\alpha e^{-j\varepsilon^2}}{2^{j+1}}\notag\\
  &>1-e^{-\varepsilon^2} + e^{-\varepsilon^2} \left(
    u^\varepsilon (y_1) + \frac{\alpha}{4}
  \right)+\frac{\alpha e^{-2\varepsilon^2}}{8}+
  \sum_{j=3}^\infty \frac{\alpha e^{-j\varepsilon^2}}{2^{j+1}}\notag\\
  &>1-e^{-2\varepsilon^2} + e^{-2\varepsilon^2}\left(
    u^\varepsilon (y_2)+\frac{\alpha}{8} 
  \right)+\frac{\alpha e^{-3\varepsilon^2}}{16}
  +\sum_{j=4}^\infty \frac{\alpha e^{-j\varepsilon^2}}{2^{j+1}}\notag\\
  &>\cdots 
  >1-e^{-j\varepsilon^2} + e^{-j\varepsilon^2}\left(
    u^\varepsilon (y_j)+\frac{\alpha}{2^{j+1}}
  \right)+\frac{\alpha e^{-(j+1)\varepsilon^2}}{2^{j+2}}
  +\sum_{k=j+2}^\infty \frac{\alpha e^{-k\varepsilon^2}}{2^{k+1}}\notag
\end{align}
and 
\begin{align}\label{eq:consistency-of-games-pf2}
  u^\varepsilon (x) &\leq 1-e^{-\varepsilon^2} + e^{-\varepsilon^2}u^\varepsilon (y_1)
  \leq 1-e^{-2\varepsilon^2} + e^{-2\varepsilon^2}u^\varepsilon (y_2)\\
  &\leq \cdots \leq 1-e^{-j\varepsilon^2} + e^{-j\varepsilon^2}u^\varepsilon (y_j)\notag
\end{align}
for all $1\leq j<N$.
When $x_j\in\Omega$ holds infinitely, 
we have $1-\alpha \leq u^\varepsilon (x)\leq 1$ by taking limit of $j\to\infty$ 
in (\ref{eq:consistency-of-games-pf1}) and (\ref{eq:consistency-of-games-pf2}).
Then, we obtain $u^\varepsilon (x)=1$ by letting $\alpha \to 0$.
Otherwise, by calculating (\ref{eq:consistency-of-games-pf1}) and (\ref{eq:consistency-of-games-pf2}) 
up to the $N$-th step, we have
\begin{align*}
  1-e^{-N\varepsilon^2} + e^{-N\varepsilon^2}g(y_N)-\alpha
  < u^\varepsilon(x)
  \leq 1-e^{-N\varepsilon^2} + e^{-N\varepsilon^2}g(y_N).
\end{align*}
\end{proof}

\subsection{Convergence of the value functions}

\begin{prop}
Assume that 
$(\mathrm{A}\ref{assum:homogenity})$-$(\mathrm{A}\ref{assum:rank})$
hold.
Let $\Omega\subset \R^n$ be an open set 
and $g_1,g_2$ be bounded and uniformly contiuous functions defined on a neighborhood 
of $\partial \Omega$. 
Let $u_1^\varepsilon, u_2^\varepsilon$ be the functions determined by 
\begin{equation*}
  u_i^\varepsilon (x)
  = \inf_{(\mathbf{v},\mathbf{w})\in \mathcal{D}}\max_{\mathbf{b}\in\mathcal{S}}
  \left\{
  \begin{aligned}
    &1-e^{-\varepsilon^2}+e^{-\varepsilon^2}
      u^\varepsilon_i (x+\delta^\varepsilon(\mathbf{v},\mathbf{w},\mathbf{b}))
      &&\text{if}\ x+\delta^\varepsilon(\mathbf{v},\mathbf{w},\mathbf{b}) \in \Omega,\\
      &1-e^{-\varepsilon^2}+e^{-\varepsilon^2}
      g_i (x+\delta^\varepsilon(\mathbf{v},\mathbf{w},\mathbf{b}))
      &&\text{if}\ x+\delta^\varepsilon(\mathbf{v},\mathbf{w},\mathbf{b})\not \in \Omega
  \end{aligned}
  \right.
\end{equation*}
in the same way as in Propositon $\ref{prop:fixed-point}$ and let 
\begin{equation*}
  \overline{u}_i(x)
  = \limsup_{\substack{\Omega \ni y\to x \\ \varepsilon \to 0}} u_i^\varepsilon (y),
  \quad 
  \underline{u}_i(x)
  = \liminf_{\substack{\Omega \ni y\to x \\ \varepsilon \to 0}} u_i^\varepsilon (y),
  \quad x\in\overline{\Omega}
\end{equation*}
for $i=1,2$. 
Then, $g_1|_{\partial\Omega}=g_2|_{\partial\Omega}$ implies 
$\overline{u}_1=\overline{u}_2$ and $\underline{u}_1=\underline{u}_2$.
\end{prop}

\begin{proof}
First, we show that 
for all $x\in\Omega$, $\varepsilon>0$ and $\alpha>0$,
there is a sequence 
$\{(\mathbf{v}_j,\mathbf{w}_j,\mathbf{b}_j)\}_{j=1}^N$ $(N\in\N\cup\{\infty\})$ 
satisfying
\begin{align}\label{eq:boundary-consistency-pf2}
  \left\{
  \begin{aligned}
  &u_1^\varepsilon (x)>1-e^{-N\varepsilon^2}+e^{-N\varepsilon^2}g_1(y_N)-\alpha,
  \quad u_2^\varepsilon (x) \leq 1-e^{-N\varepsilon^2}+e^{-N\varepsilon^2}g_2(y_N)
  &&\text{if }N<\infty,\\
  &u_1^\varepsilon (x)>1-\alpha,\quad u_2^\varepsilon (x)\leq 1
  &&\text{if } N=\infty,
  \end{aligned}
  \right.
\end{align}
where $y_j=y_{j-1}+\delta^\varepsilon (\mathbf{v}_j,\mathbf{w}_j,\mathbf{b}_j)$
and $y_0=x$.
We can construct such a sequence using the same procedure as in Proposition \ref{consistency-with-games}. 
Define $\widehat{u}_i^\varepsilon$ on an open neighborhood of 
$\overline{\Omega}$ by 
\begin{align*}
  \widehat{u}_i^\varepsilon (x)
  =\left\{
  \begin{aligned}
    & u_i^\varepsilon (x) &&\text{if } x\in\Omega,\\
    & g_i(x) && \text{if } x\not\in \Omega
  \end{aligned}
  \right.
\end{align*}
for $i=1,2$.
Fix $x\in\Omega,\ \varepsilon>0$ and $\alpha>0$
and choose $(\mathbf{v}_1,\mathbf{w}_1)\in\mathcal{D}$ such that 
\begin{align*}
  u_1(x)+\frac{\alpha}{2}
  > \max_\mathbf{b\in\mathcal{S}}
  \left\{
    1-e^{-\varepsilon^2}+e^{-\varepsilon^2}\widehat{u}_1
    (x+\delta^\varepsilon (\mathbf{v}_1,\mathbf{w}_1,\mathbf{b}))
  \right\}.
\end{align*}
Then, choose $\mathbf{b}_1\in\mathcal{S}$ such that 
\begin{align*}
  \widehat{u}_2^\varepsilon (x+\delta^\varepsilon (\mathbf{v}_1,\mathbf{w}_1,\mathbf{b}_1))
  =\max_{\mathbf{b}\in\mathcal{S}}
  \widehat{u}_2^\varepsilon (x+\delta^\varepsilon (\mathbf{v}_1,\mathbf{w}_1,\mathbf{b}))
\end{align*}
and set $y_1 = x +\delta^\varepsilon (x+\delta^\varepsilon (\mathbf{v}_1,\mathbf{w}_1,\mathbf{b}_1))$.
Iteratively for each $j=2,3,\ldots$, 
as long as $y_{j-1}\in \Omega$ holds,
choose $(\mathbf{v}_j,\mathbf{w}_j)$ with 
\begin{align*}
  u_1^\varepsilon (y_{j-1})+\frac{\alpha}{2^{j}}
  > \max_{\mathbf{b}\in\mathcal{S}}\left\{ 1-e^{-\varepsilon^2}+
  e^{-\varepsilon^2}\widehat{u}_1^\varepsilon (y_{j-1}+\delta^\varepsilon (\mathbf{v}_j,\mathbf{w}_j,\mathbf{b}))
  \right\}
\end{align*}
and pick $\mathbf{b}_j$ with 
\begin{align*}
  \widehat{u}_2^\varepsilon(y_{j-1}+\delta^\varepsilon(\mathbf{v}_j,\mathbf{w}_j,\mathbf{b}_j))
  =\max_{\mathbf{b}\in\mathcal{S}}\widehat{u}_2^\varepsilon(y_{j-1}+\delta^\varepsilon (\mathbf{v}_j,\mathbf{w}_j,\mathbf{b})).
\end{align*}
When $y_j:=y_{j-1}+\delta^\varepsilon (\mathbf{v}_j,\mathbf{w}_j,\mathbf{b}_j)
\not\in\Omega$ occurs for the first time for some $j$, 
we let $N=j$.
Then, by the same calculation as (\ref{eq:consistency-of-games-pf1}) and (\ref{eq:consistency-of-games-pf2}),
we have (\ref{eq:boundary-consistency-pf2}).

Take $x\in\overline{\Omega}$ arbitrarily.
Then, there exist sequences $\{x_k\}\subset \Omega$ and $\{\varepsilon_k\}$
such that $ x_k\to x$, $\varepsilon \to 0$, 
$u_2^{\varepsilon_k}(x_k)\to\overline{u}_2(x)$.
We take a sequence $\{(\mathbf{v}_{k,j},\mathbf{w}_{k,j},
\mathbf{b}_{k,j})\}_{j=1}^{N_k}$ for each $k$ such that 
\begin{align}\label{eq:boundary-consistency-pf1}
  \left\{
  \begin{aligned}
  &u_1^{\varepsilon_k} (x_k)>1-e^{-N_k\varepsilon_k^2}
  +e^{-N_k\varepsilon_k^2}g_1(y_{k,N_k})-\frac{1}{k},\\
  &u_2^{\varepsilon_k} (x_k) 
  \leq 1-e^{-N_k\varepsilon_k^2}+e^{-N_k\varepsilon_k^2}g_2(y_{k,N_k})
  \end{aligned}
  \right.
\end{align}
where $y_{k,0}=x_k$, $y_{k,j}=y_{k,j-1}+\delta^{\varepsilon_k}(\mathbf{v}_{k,j},\mathbf{w}_{k,j},
\mathbf{b}_{k,j})$.
Using (\ref{eq:boundary-consistency-pf1}) and $|\delta^\varepsilon|=O(\varepsilon)$,
we have
\begin{align*}
  u_2^{\varepsilon_k}(x_k)
  &\leq 1-e^{-N_k \varepsilon_k^2} + e^{-N_k \varepsilon_k^2}g_2(y_{k,N_k})\\
  &\leq 1-e^{-N_k \varepsilon_k^2} + e^{-N_k \varepsilon_k^2}g_1(y_{k,N_k})
  +e^{-N_k \varepsilon_k^2}|g_1(y_{k,N_k})-g_2(y_{k,N_k})|\\
  &\leq u_1^{\varepsilon_k}(x_k)+\frac{1}{k}
  +e^{-N_k\varepsilon_k^2} \sup_{B(\partial\Omega,C\varepsilon_k)}|g_1-g_2|.
\end{align*}
By the definition of $\overline{u}_i$ and uniform continuity of $g_i$, it follows that 
\begin{align*}
  \overline{u}_2(x)
  = \lim_{k\to\infty} u^{\varepsilon_k}(x_k)
  \leq \limsup_{k\to\infty}\left\{
  u_1^{\varepsilon_k}(x_k)+\frac{1}{k}
  +e^{-N_k\varepsilon_k^2} \sup_{B(\partial\Omega,C\varepsilon_k)}|g_1-g_2|
  \right\}
  \leq \overline{u}_1(x)
\end{align*}
if $g_1|_{\partial \Omega}=g_2|_{\partial\Omega}$.
By similar arguments and symmetry with $i=1,2$,
we can get $\underline{u}_2(x)\leq \underline{u}_1(x)$, 
$\overline{u}_1(x)\leq \overline{u}_2(x)$ and $\underline{u}_1(x)\leq \underline{u}_2(x)$.
\end{proof}

\begin{thm}\label{thm:game-conv}
Assume $(\mathrm{A}\ref{assum:homogenity})$-$(\mathrm{A}\ref{assum:rank})$, 
$\Omega\subset \R^n$ be an open set 
and $g\in C(\partial\Omega)$ be bounded and uniformly continuous. 
Let $u^\varepsilon$ be the function determined by
\begin{equation*}
  u^\varepsilon (x)
  = \inf_{(\mathbf{v},\mathbf{w})\in \mathcal{D}}\max_{\mathbf{b}\in\mathcal{S}}
  \left\{
  \begin{aligned}
    &1-e^{-\varepsilon^2}+e^{-\varepsilon^2}
      u^\varepsilon (x+\delta^\varepsilon(\mathbf{v},\mathbf{w},\mathbf{b}))
      &&\text{if}\ x+\delta^\varepsilon(\mathbf{v},\mathbf{w},\mathbf{b}) \in \Omega,\\
      &1-e^{-\varepsilon^2}+e^{-\varepsilon^2}
      g (x+\delta^\varepsilon(\mathbf{v},\mathbf{w},\mathbf{b}))
      &&\text{if}\ x+\delta^\varepsilon(\mathbf{v},\mathbf{w},\mathbf{b})\not \in \Omega 
  \end{aligned}
  \right.
\end{equation*}
for a continuous extention of $g$ to a neighborhood of $\partial\Omega$. 
And let
\begin{equation*}
  \overline{u}(x)
  = \limsup_{\substack{\Omega \ni y\to x \\ \varepsilon \to 0}} u^\varepsilon (y),
  \quad 
  \underline{u}(x)
  = \liminf_{\substack{\Omega \ni y\to x \\ \varepsilon \to 0}} u^\varepsilon (y),
  \quad x\in\overline{\Omega}.
\end{equation*}
Then, the function $\overline{u}$ (resp., $\underline{u}$) is 
upper semicontinuous (resp., lower semicontinuous) on $\overline{\Omega}$ 
and a viscosity subsolution (resp., supersolution) of
\begin{align*}
  \left\{
  \begin{aligned}
    u+F(Du,D^2u)&=1&&\text{in }\Omega,\\
    u&=g&&\text{on }\partial\Omega.
  \end{aligned}
  \right.
\end{align*}
\end{thm}

We need the following lemma in our proof of Theorem \ref{thm:game-conv}
and its proof is similar to \cite[Lemma 2.3]{MR2200259}.

\begin{lem}\label{lem:lem-game-conv}
In the same assumption of Theorem $\ref{thm:game-conv}$, 
let $\phi$ be a smooth function defined on $\Omega$. 
\begin{enumerate}[$(1)$]
  \item For every $x\in\Omega$ with $D\phi(x)\neq 0$,
  \begin{equation*}
    \min_{(\mathbf{v},\mathbf{w})\in \mathcal{D}}\max_{\mathbf{b}\in\mathcal{S}}\phi\left(
      x+\delta^\varepsilon (\mathbf{v},\mathbf{w},\mathbf{b})
    \right)
    \leq\phi(x) - \varepsilon^2 F(D\phi(x),D^2\phi(x)) + o(\varepsilon^2).
  \end{equation*}
  \item For every $x\in\Omega$ and $(\mathbf{v},\mathbf{w})\in\mathcal{D}$,
  \begin{equation*}
    \max_{\mathbf{b}\in\mathcal{S}}\phi\left(
      x+\delta^\varepsilon (\mathbf{v},\mathbf{w},\mathbf{b})
    \right)
    \geq \phi(x)+\varepsilon^2 \left\{
      \tr[\sigma(v^1){}^t\!\sigma(v^1)D^2\phi(x)]-c(v^1)|D\phi(x)|
    \right\}+o(\varepsilon^2).
  \end{equation*}
  \item For every $x\in\Omega$ with $D\phi(x)\neq 0$, 
  \begin{equation*}
    \min_{(\mathbf{v},\mathbf{w})\in \mathcal{D}}\max_{\mathbf{b}\in\mathcal{S}}\phi\left(
      x+\delta^\varepsilon (\mathbf{v},\mathbf{w},\mathbf{b})
    \right)
    \geq\phi(x) - \varepsilon^2 F(D\phi(x),D^2\phi(x)) +\left(\frac{1}{|D\phi(x)|}+1\right)o(\varepsilon^2).
  \end{equation*}
\end{enumerate}
\end{lem}

\begin{proof}
By the Taylor expansion of $\phi$ at $x$, 
for small $\varepsilon>0$, we have
\begin{align}\label{eq:lem-game-conv-pf1}
  \phi\left(
    x+\delta^\varepsilon (\mathbf{v},\mathbf{w},\mathbf{b})\right)
  =\phi(x)&+\sqrt{2}\varepsilon \sum_{i=1}^m b^i 
  \langle D\phi(x),\sigma(v^1)w^i\rangle\\
  &+\varepsilon^2 \sum_{i,j=1}^m b^ib^j
  \langle D^2\phi(x)\sigma(v^1)w^i,\sigma(v^1)w^j\rangle\notag\\
  &+\varepsilon^2 c(v^1)\langle D\phi(x),v^2 \rangle
  +o(\varepsilon^2).\notag
\end{align}
If $D\phi(x)\neq 0$, choosing $\mathbf{v}$ as 
\begin{equation*}
  \widehat{v}^1 = \frac{D\phi(x)}{|D\phi(x)|},\quad 
  \widehat{v}^2 = -\frac{D\phi(x)}{|D\phi(x)|}
\end{equation*}
and $\mathbf{w}$ as $\{\widehat{w}^i\}$ : the unit eigenvectors 
of the $(m\times m)$-symmetric matrix 
${}^t\!\sigma(D\phi(x))D^2\phi(x)\sigma(D\phi(x))$ 
with respect to each eigenvalue that are orthogonal to each other,
we obtain 
\begin{align*}
  \langle 
  D\phi(x),\sigma(\widehat{v}^1)\widehat{w}^i 
  \rangle 
  = \langle 
    D\phi(x),\sigma(D\phi(x))\widehat{w}^i 
  \rangle
  = 0\quad \text{for } i=1,\ldots,m
\end{align*}
by (A\ref{assum:homogenity}) and (A\ref{assum:rank}).
We also obtain
\begin{align}\label{eq:lem-game-conv-pf3}
  \sum_{i,j=1}^m b^i b^j \langle
    D^2\phi(x)\sigma(\widehat{v}^1)\widehat{w}^i,
    \sigma(\widehat{v}^1)\widehat{w}^j
  \rangle
  &=\sum_{i=1}^m \langle 
    {}^t\!\sigma(D\phi(x))D^2\phi(x)\sigma(D\phi(x))\widehat{w}^i,
    \widehat{w}^i
  \rangle\\
  &=\tr \left[\sigma(D\phi(x)){}^t\!\sigma(D\phi(x))D^2\phi(x)\right].\notag
\end{align}
by (A\ref{assum:homogenity}) and 
the fact that $\{\widehat{w}^i\}$ is an orthonormal basis 
consisting of eigenvectors of\\ ${}^t\!\sigma(D\phi(x))D^2\phi(x)\sigma(D\phi(x))$ and 
$\tr{XY}=\tr{YX}$ holds for any two matrices $X,Y$ 
whose products of both sides are defined.
Then, it follows that 
\begin{align*}
  &\min_{(\mathbf{v},\mathbf{w})\in \mathcal{D}}\max_{\mathbf{b}\in\mathcal{S}}\phi\left(
    x+\delta^\varepsilon (\mathbf{v},\mathbf{w},\mathbf{b})
  \right)
  \leq \max_{\mathbf{b}\in\mathcal{S}}\phi(x+\delta^\varepsilon(\widehat{\mathbf{v}},\widehat{\mathbf{w}},\mathbf{b}))\\
  &\leq \max_{\mathbf{b}\in\mathcal{S}}\left\{\phi(x)
  +\varepsilon^2 \sum_{i,j=1}^m b^i b^j \langle
    D^2\phi(x)\sigma(\widehat{v}^1)\widehat{w}^i,
    \sigma(\widehat{v}^1)\widehat{w}^j
  \rangle
  -\varepsilon^2 c(D\phi(x))
  +o(\varepsilon^2)\right\}\\
  &\leq \phi(x) - \varepsilon^2 F(D\phi(x),D^2\phi(x))+o(\varepsilon^2)
\end{align*}
by (\ref{eq:lem-game-conv-pf3}).
This gives inequality (1). 

We would like to show the rest.
First, calculating a part of terms in (\ref{eq:lem-game-conv-pf1}) as
\begin{align*}
  &\sqrt{2}\varepsilon \sum_{i=1}^m b^i 
  \langle D\phi(x), \sigma(v^1)w^i \rangle
  +\varepsilon^2 \sum_{i,j=1}^m b^ib^j
  \langle D^2\phi(x)\sigma(v^1)w^i,\sigma(v^1)w^j \rangle\\
  &=\sqrt{2}\varepsilon b^1 \langle D\phi(x),\sigma(v^1)w^1 \rangle
  +\varepsilon^2 \sum_{i=1}^m \langle D^2 \phi(x)\sigma(v^1)w^i,\sigma(v^1)w^i\rangle\\
  &\quad +\sum_{i=2}^m b^i \left(
    \sqrt{2}\varepsilon \langle D\phi(x),\sigma(v^1)w^i \rangle
    + 2\varepsilon^2 \sum_{j=1}^{i-1} b^j \langle
      D^2\phi(x)\sigma(v^1)w^i,\sigma(v^1)w^j
    \rangle
  \right),
\end{align*}
we can choose $b^1,\ldots,b^m$ in this order 
so that each term summarized by $b^i$ is positive. 
This implies 
\begin{align*}
  &\max_{\mathbf{b}\in\mathcal{S}}\phi\left( x+\delta^\varepsilon(\mathbf{v},\mathbf{w},\mathbf{b}) \right)\\
  &\geq \phi(x)+\sqrt{2}\varepsilon|\langle D\phi(x),\sigma(v^1)w^1\rangle|
  +\varepsilon^2\sum_{i=1}^m \langle D^2\phi(x)\sigma(v^1)w^i,\sigma(v^1)w^i\rangle\\
  &\quad\quad\quad +\varepsilon^2 \langle D\phi(x),c(v^1)v^2\rangle 
  + o(\varepsilon^2)
\end{align*}
for all $(\mathbf{v},\mathbf{w})\in\mathcal{D}$.
Similarly for $i=2,\ldots,m$, we have
\begin{align}\label{eq:lem-game-conv-pf2}
  &\max_{\mathbf{b}\in\mathcal{S}}\phi\left( x+\delta^\varepsilon(\mathbf{v},\mathbf{w},\mathbf{b}) \right)\\
  &\geq \phi(x)+\max_{i=1,\ldots,m}\sqrt{2}\varepsilon|\langle D\phi(x),\sigma(v^1)w^i\rangle|
  +\varepsilon^2\sum_{i=1}^m \langle D^2\phi(x)\sigma(v^1)w^i,\sigma(v^1)w^i\rangle\notag\\
  &\quad\quad\quad +\varepsilon^2 \langle D\phi(x),c(v^1)v^2\rangle 
  + o(\varepsilon^2)\notag
\end{align}
for all $(\mathbf{v},\mathbf{w})\in\mathcal{D}$.
Similarly to (\ref{eq:lem-game-conv-pf3}),
we can compute that $\sum_{i}\langle D^2\phi (x)\sigma(v^1)w^i,\sigma(v^1)w^i\rangle=\tr [\sigma(v^1){}^t\!\sigma(v^1)D^2\phi(x)]$. 
Thus, (2) straightly follows from (\ref{eq:lem-game-conv-pf2}).

To derive the inequality (3), we assume $D\phi(x)\neq 0$, and then we see the following claim.\\
\textbf{Claim.} 
Let $\Lambda:=\max_{|v|=1}|\langle D^2\phi(x)v,v \rangle|$. 
Then, the following two properties hold for every $(\mathbf{v},\mathbf{w})\in\mathcal{D}$. 
\begin{enumerate}[(i)]
  \item If $\max_{i=1,\ldots,m}\sqrt{2}\varepsilon|\langle D\phi(x),\sigma(v^1)w^i\rangle|
  \geq 2\varepsilon^2 (m\Lambda+|D\phi(x)|\lVert c \rVert_{L^\infty(\partial B(0,1))})$, 
  then 
  \begin{align*}
    &\max_{i=1,\ldots,m}\sqrt{2}\varepsilon |\langle D\phi(x),\sigma(v^1)w^i \rangle|
    +\varepsilon^2\sum_{i=1}^m \langle D^2\phi(x)\sigma(v^1)w^i,\sigma(v^1)w^i\rangle
    +\varepsilon^2c(v^1)\langle D\phi(x),v^2 \rangle\\
    &\geq \varepsilon^2 \tr [\sigma(D\phi(x)){}^t\!\sigma(D\phi(x))D^2\phi(x)]
    +\varepsilon^2 c\left(\frac{D\phi(x)}{|D\phi(x)|}\right) \left\langle D\phi(x), v^2\right\rangle.
  \end{align*}
  \item If $\max_{i=1,\ldots,m}\sqrt{2}\varepsilon|\langle D\phi(x),\sigma(v^1)w^i\rangle|
  < 2\varepsilon^2 (m\Lambda+|D\phi(x)|\lVert c \rVert_{L^\infty(\partial B(0,1))})$, 
  then 
  \begin{align*}
    &\max_{i=1,\ldots,m}\sqrt{2}\varepsilon |\langle D\phi(x),\sigma(v^1)w^i \rangle|
    +\varepsilon^2\sum_{i=1}^m \langle D^2\phi(x)\sigma(v^1)w^i,\sigma(v^1)w^i\rangle
    +\varepsilon^2c(v^1)\langle D\phi(x),v^2 \rangle\\
    &\geq \varepsilon^2 \tr [\sigma(D\phi(x)){}^t\!\sigma(D\phi(x))D^2\phi(x)]
    +\varepsilon^2 c\left(\frac{D\phi(x)}{|D\phi(x)|}\right)\left\langle D\phi(x), v^2\right\rangle
    +o(\varepsilon^2).
  \end{align*}
\end{enumerate}
\textit{Proof of Claim.} (i) is clear by the setting of $\Lambda$. 
We need to check (ii). 
Assume \\$\max_{i=1,\ldots,m}\sqrt{2}\varepsilon|\langle D\phi(x),\sigma(v^1)w^i\rangle|
  < 2\varepsilon^2 (m\Lambda+|D\phi(x)|\lVert c \rVert_{L^\infty(\partial B(0,1))})$. Then, 
\begin{align*}
  \dist(\sigma(v^1)w^i,\langle D\phi(x) \rangle^\perp)
  =\left|\left\langle\frac{D\phi(x)}{|D\phi(x)|},\sigma(v^1)w^i\right\rangle\right|
  <C\left(\frac{\Lambda}{|D\phi(x)|}+1\right)\varepsilon
\end{align*}
for $i=1,\ldots,m$. 
Thus, assumptions $(\mathrm{A}\ref{assum:Lipschitz-conti})$, $(\mathrm{A}\ref{assum:evenness})$
and $(\mathrm{A}\ref{assum:rank})$ 
of $\sigma$ implies 
\begin{align*}
  \left|v^1-\frac{D\phi(x)}{|D\phi(x)|}\right|\quad \text{or}\quad
  \left|v^1+\frac{D\phi(x)}{|D\phi(x)|}\right|\leq C\left(\frac{\Lambda}{|D\phi(x)|}+1\right)\varepsilon,
\end{align*}
and thus 
\begin{align*}
  &\varepsilon^2\sum_{i=1}^m \langle D^2\phi(x)\sigma(v^1)w^i,\sigma(v^1)w^i\rangle
  +\varepsilon^2c(v^1)\langle D\phi(x),v^2 \rangle\\
  &\geq \varepsilon^2 \tr [\sigma(D\phi(x)){}^t\!\sigma(D\phi(x))D^2\phi(x)]
  +\varepsilon^2 c\left(\frac{D\phi(x)}{|D\phi(x)|}\right)\langle D\phi(x), v^2\rangle
  -C\left(\frac{\Lambda}{|D\phi(x)|}+1\right)\varepsilon^3.
\end{align*}
We obtain (ii) from this equality.

Finally, the following inequality holds for all $(\mathbf{v},\mathbf{w})\in\mathcal{D}$.
\begin{align*}
  &\max_{\mathbf{b}\in\mathcal{S}}\phi\left( x+\delta^\varepsilon(\mathbf{v},\mathbf{w},\mathbf{b}) \right)\\
  &\geq \phi(x)+\max_{i=1,\ldots,m}\sqrt{2}\varepsilon |\langle D\phi(x),\sigma(v^1)w^i \rangle|
  +\varepsilon^2\sum_{i=1}^m \langle D^2\phi(x)\sigma(v^1)w^i,\sigma(v^1)w^i\rangle\\
  &\quad\quad\quad +\varepsilon^2 c(v^1)\langle D\phi(x),v^2\rangle + o(\varepsilon^2)\\
  &\geq \phi(x)+\varepsilon^2 \tr [\sigma(D\phi(x)){}^t\!\sigma(D\phi(x))D^2\phi(x)]
  +\varepsilon^2 c\left(\frac{D\phi(x)}{|D\phi(x)|}\right)\langle D\phi(x), v^2\rangle
  +o(\varepsilon^2).
\end{align*}
Taking minimum for $(\mathbf{v},\mathbf{w})$, we get a conclusion.
\end{proof}

\begin{proof}[Proof of Theorem $\ref{thm:game-conv}$]
Upper and lower semi continuity of $\overline{u}$ and $\underline{u}$, 
respectively, is clear. \\
\textbf{Subsolution test for }$\overline{u}$ \textbf{.} 
We first show that $\overline{u}+F_*(D\overline{u},D^2\overline{u})\leq 1$ 
in $\Omega$. 
Fix $x_0\in\Omega$ and consider $r>0$ 
and a smooth function $\phi$ such that
\begin{equation}\label{eq:game-conv-pf1}
  \overline{u}(x_0)=\phi(x_0),\ B(x_0,r)\subset\Omega,\ 
  \overline{u}-\phi < 0 \text{ in } B(x_0,r)\setminus\{x_0\}.
\end{equation}
There exist sequences $\{\widetilde{x}_k\}\subset B(x_0,r)$ 
and $\{\varepsilon_k\}$ such that
\begin{equation*}
  \widetilde{x}_k\to x_0,\ \varepsilon_k\to 0,\ 
  u^{\varepsilon_k}(\widetilde{x}_k)\to \overline{u}(x_0) \text{ as }k\to\infty
\end{equation*}
by the definition of $\overline{u}$. 
For every $k$, choose $x_k\in B(x_0,r)$ satisfying
\begin{equation*}
  (u^{\varepsilon_k}-\phi)(x_k)
  >\sup_{B(x_0,r)}(u^{\varepsilon_k}-\phi)-\varepsilon_k^3
  \geq (u^{\varepsilon_k}-\phi)(\widetilde{x}_k)-\varepsilon_k^3.
\end{equation*}
Since $\overline{u}-\phi$ attains its local strict maximum at $x_0$, 
taking a subsequence if necessary, $\{x_k\}$ satisfies 
\begin{equation*}
  x_k\to x_0,\ u^{\varepsilon_k}(x_k)\to \overline{u}(x_0).
\end{equation*}
In this way, we got the existence of sequences $\{x_k\}$ and $\{\varepsilon_k\}$ 
satisfying
\begin{equation}\label{eq:game-conv-pf2}
  x_k\to x_0,\ 
  \varepsilon_k\to 0,\ 
  u^{\varepsilon_k}(x_k)\to \overline{u}(x_0),\ 
  (u^{\varepsilon_k}-\phi)(x_k) > \sup_{B(x_0,r)}(u^{\varepsilon_k}-\phi)-\varepsilon_k^3.
\end{equation}
By (\ref{eq:game-conv-pf2}) and $|\delta^{\varepsilon}(\mathbf{v},\mathbf{w},\mathbf{b})|=O(\varepsilon)$, 
the inequality 
\begin{equation*}
  u^{\varepsilon_k}(x_k+\delta^{\varepsilon_k}(\mathbf{v},\mathbf{w},\mathbf{b}))
  \leq\phi(x_k+\delta^{\varepsilon_k}(\mathbf{v},\mathbf{w},\mathbf{b}))
  +(u^{\varepsilon_k}-\phi)(x_k)+\varepsilon_k^3
\end{equation*}
holds for all $(\mathbf{v},\mathbf{w})\in\mathcal{D}$, $\mathbf{b}\in\mathcal{S}$ 
if $k$ is sufficiently large.
Substituting the above into the dynamic programming principle of $u^{\varepsilon_k}$, 
we have
\begin{align}\label{eq:game-conv-pf3}
  u^{\varepsilon_k}(x_k)
  &= \inf_{(\mathbf{v},\mathbf{w})\in \mathcal{D}}\max_{\mathbf{b}\in\mathcal{S}}
  \left\{ 1-e^{-\varepsilon_k^2} +e^{-\varepsilon_k^2}
  u^{\varepsilon_k}(x_k+\delta^{\varepsilon_k}(\mathbf{v},\mathbf{w},\mathbf{b}))\right\}\\
  &\leq \min_{(\mathbf{v},\mathbf{w})\in \mathcal{D}}\max_{\mathbf{b}\in\mathcal{S}}
  \left\{ 1-e^{-\varepsilon_k^2} +e^{-\varepsilon_k^2}
  \phi(x_k+\delta^{\varepsilon_k}(\mathbf{v},\mathbf{w},\mathbf{b}))\right\}
  +e^{-\varepsilon_k^2}\left\{(u^{\varepsilon_k}-\phi)(x_k)+\varepsilon_k^3\right\}\notag
\end{align}
for sufficiently large $k$. 
Now, we need to consider two cases since $F(p,X)$ is singular at $p=0$. \\
\textbf{Case 1.} 
The case that $D\phi(x_0)\neq 0$. 
In this case, $D\phi(x_k)\neq 0$ for every large $k$ and 
then Lemma \ref{lem:lem-game-conv} (1) implies
\begin{align*}
  \min_{(\mathbf{v},\mathbf{w})\in \mathcal{D}}\max_{\mathbf{b}\in\mathcal{S}}\phi\left(
      x_k+\delta^{\varepsilon_k} (\mathbf{v},\mathbf{w},\mathbf{b})
    \right)
    \leq\phi(x_k) - \varepsilon_k^2 F(D\phi(x_k),D^2\phi(x_k)) + o(\varepsilon_k^2).
\end{align*}
Substituting this into (\ref{eq:game-conv-pf3}), 
we have 
\begin{equation}\label{eq:game-conv-pf4}
  (1-e^{-\varepsilon_k^2})u^{\varepsilon_k}(x_k)
  \leq 1-e^{-\varepsilon_k^2} 
  -\varepsilon_k^2e^{-\varepsilon_k^2}F(D\phi(x_k),D^2\phi(x_k))
  +o(\varepsilon_k^2).
\end{equation}
Dividing both sides by $\varepsilon_k^2$ and letting $k\to \infty$, 
we conclude that 
\begin{align*}
  \overline{u}(x_0)\leq 1-F(D\phi(x_0),D^2\phi(x_0)).
\end{align*}
\textbf{Case 2.} The case that $D\phi(x_0)=0$. 
If there is a subsequence $\{k_l\}$ with $D\phi(x_{k_l})\neq 0$ 
and $D\phi(x_{k_l})\to 0$, we apply (\ref{eq:game-conv-pf4}) 
for this sub sequence and conclude that
\begin{align*}
  \overline{u}(x_0)+F_*(0,D^2\phi(x_0))
  &\leq \liminf_{l\to \infty}
  \left\{\frac{1-e^{-\varepsilon_{k_l}^2}}{\varepsilon_{k_l}^2}u^{\varepsilon_{k_l}}(x_{k_l})
  +e^{-\varepsilon_{k_l}^2}F(D\phi(x_{k_l}),D^2\phi(x_{k_l}))\right\}\\
  &\leq \lim_{l\to\infty} \left\{\frac{1-e^{-\varepsilon_{k_l}^2}}{\varepsilon_{k_l}^2} + o(1)\right\}
  =1.
\end{align*}
On the other hand, if $D\phi(x_k)\equiv 0$, then the Taylor expansion implies
\begin{align}\label{eq:game-conv-pf5}
  &\min_{(\mathbf{v},\mathbf{w})\in \mathcal{D}}\max_{\mathbf{b}\in\mathcal{S}}\phi\left(
      x_k+\delta^{\varepsilon_k} (\mathbf{v},\mathbf{w},\mathbf{b})
    \right)\\
  &=\min_{\substack{|v^1|=1\\ \mathbf{w}\in\mathcal{D}_2}}\max_{\mathbf{b}\in\mathcal{S}}\left\{ \phi(x_k)+
    \varepsilon_k^2\sum_{i,j=1}^m b^ib^j \langle 
      D^2\phi(x_k)\sigma(v^1)w^i,\sigma(v^1)w^j
    \rangle +o(\varepsilon_k^2)
    \right\}\notag
\end{align}
For every $v^1$, choosing a set of eigenvectors $\{\widehat{w}^i\}$ 
of the $(m\times m)$-symmetric matrix\\
${}^t\!\sigma(v^1)D^2\phi(x_k)\sigma(v^1)$ 
with respect to each eigenvalue which forms an orthonormal basis of $\R^m$, 
we have 
\begin{align}\label{eq:game-conv-pf6}
  &\tr[\sigma(v^1){}^t\!\sigma(v^1)D^2\phi(x_k)]
  =\max_{\mathbf{b}\in\mathcal{S}}
  \sum_{i,j=1}^m b^i b^j \langle {}^t\! \sigma(v^1)D^2\phi(x_k)\sigma(v^1)\widehat{w}^i,
  \widehat{w}^j\rangle\\
  &\geq \min_{\substack{|v'^1|=1\\ \mathbf{w}'\in\mathcal{D}_2}}
  \max_{\mathbf{b}\in\mathcal{S}}
  \sum_{i,j=1}^m b^ib^j
  \langle D^2\phi(x_k)\sigma(v'^1)w'^i,\sigma(v'^1)w'^j\rangle.\notag
\end{align}
We combining (\ref{eq:game-conv-pf5}) with (\ref{eq:game-conv-pf6}) 
and Substituting it into (\ref{eq:game-conv-pf3}), 
it follows that
\begin{equation*}
  (1-e^{-\varepsilon_k^2})u^{\varepsilon_k}(x_k)
  \leq 1-e^{-\varepsilon_k^2} 
  +\varepsilon_k^2e^{-\varepsilon_k^2}\tr[\sigma(v^1){}^t\!\sigma(v^1)D^2\phi(x_k)]
  +o(\varepsilon_k^2),\quad \forall v^1\in\partial B(0,1).
\end{equation*}
Note that $\sigma(\cdot)$ is $0$-homogenous function and then we have
\begin{align*}
  \lim_{k\to \infty}\tr[\sigma(v^1){}^t\!\sigma(v^1)D^2\phi(x_k)]
  &=\lim_{k\to \infty}\left\{ 
    \tr[\sigma(\varepsilon_kv^1){}^t\!\sigma(\varepsilon_kv^1)D^2\phi(x_k)]-c(\varepsilon_k v^1)
    \right\}\\
  &\leq (-F)^*(0,D^2\phi(x_0)) = -F_*(0,D^2\phi(x_0))
\end{align*}
for all $v^1\in\partial B(0,1)$. 
Therefore, we conclude that 
\begin{equation*}
  \overline{u}(x_0)\leq 1-F_*(0,D^2\phi(x_0)).
\end{equation*}

Next, we show the sub boundary condition of $\overline{u}$.
Fix a boundary point $x_0\in\partial \Omega$ and assume 
$(\overline{u}-g)(x_0)=:\alpha>0$.
Choose $r>0$ and $\phi\in C^\infty (\overline{\Omega})$ satisfying 
\begin{equation*}
  \overline{u}(x_0)=\phi(x_0),\ 
  \overline{u}-\phi <0 \ \text{on } 
  (B(x_0,r)\cap\overline{\Omega})\setminus \{x_0\},\ 
  g < g(x_0)+\frac{\alpha}{4} \ \text{in } B(x_0,r).
\end{equation*}
Similarly to (\ref{eq:game-conv-pf2}), 
there exist sequences $\{x_k\}\subset B(x_0,r)\cap\Omega$, 
$\{\varepsilon_k\}$ such that
\begin{equation}
  x_k\to x_0,\ \varepsilon_k\to 0,\ 
  u^{\varepsilon_k}(x_k)\to \overline{u}(x_0),\ 
  (u^{\varepsilon_k}-\phi)(x_k)
  > \sup_{B(x_0,r)\cap\Omega}(u^{\varepsilon_k} -\phi)-\varepsilon_k^3.
\end{equation}
Then, $u^{\varepsilon_k}(x_k)\to \overline{u}(x_0)$ implies, 
for sufficiently large $k$, 
\begin{align*}
  g(x_0)+\frac{\alpha}{2}
  &<u^{\varepsilon_k} (x_k)\\
  &=\inf_{(\mathbf{v},\mathbf{w})\in \mathcal{D}}\max_{\mathbf{b}\in\mathcal{S}}
  \left\{
  \begin{aligned}
    &1-e^{-\varepsilon_k^2}+e^{-\varepsilon_k^2}
    u^{\varepsilon_k} (x_k+\delta^{\varepsilon_k}(\mathbf{v},\mathbf{w},\mathbf{b}))
    &&\text{if}\ x+\delta^{\varepsilon_k}(\mathbf{v},\mathbf{w},\mathbf{b})\in \Omega,\\
    &1-e^{-\varepsilon_k^2}+e^{-\varepsilon_k^2}
    g (x_k+\delta^{\varepsilon_k}(\mathbf{v},\mathbf{w},\mathbf{b}))
    &&\text{if}\ x+\delta^{\varepsilon_k}(\mathbf{v},\mathbf{w},\mathbf{b})\not\in \Omega
  \end{aligned}
  \right.\\
  &\leq \min_{(\mathbf{v},\mathbf{w})\in \mathcal{D}}\max_{\mathbf{b}\in\mathcal{S}}
  \left\{
  \begin{aligned}
    &1-e^{-\varepsilon_k^2}+e^{-\varepsilon_k^2}
    \phi (x_k+\delta^{\varepsilon_k}(\mathbf{v},\mathbf{w},\mathbf{b}))
    +e^{-\varepsilon_k^2}\left\{(u^{\varepsilon_k}-\phi)(x_k)+\varepsilon_k^3\right\}\\
    &\quad\quad\quad\quad\quad\quad\quad\quad\quad\quad\quad\quad\quad\quad\quad\quad\quad\quad
    \text{if}\ x+\delta^{\varepsilon_k}(\mathbf{v},\mathbf{w},\mathbf{b})\in \Omega,\\
    &1-e^{-\varepsilon_k^2}+e^{-\varepsilon_k^2}
    g (x_k+\delta^{\varepsilon_k}(\mathbf{v},\mathbf{w},\mathbf{b}))
    \quad\quad\text{if}\ x+\delta^{\varepsilon_k}(\mathbf{v},\mathbf{w},\mathbf{b})\not\in \Omega.
  \end{aligned}
  \right.
\end{align*}
However, if $k$ is sufficiently large, it holds that
\begin{equation*}
  1-e^{-\varepsilon_k^2}
  +e^{-\varepsilon_k^2}g(x_k+\delta^{\varepsilon_k}(\mathbf{v},\mathbf{w},\mathbf{b}))
  <\frac{\alpha}{4}+g(x_0)+\frac{\alpha}{4}=g(x_0)+\frac{\alpha}{2}.
\end{equation*}
Therefore, $\min_{(\mathbf{v},\mathbf{w})\in \mathcal{D}}\max_{\mathbf{b}\in\mathcal{S}}$ is taken only for 
$(\mathbf{v},\mathbf{w})$ and $\mathbf{b}$ 
that $x+\delta^{\varepsilon_k}(\mathbf{v},\mathbf{w},\mathbf{b})\in\Omega$. 
And by taking the appropriate extention of $\phi$ to $B(x_0,r)$, 
we have 
\begin{equation*}
  u^{\varepsilon_k} (x_k)
  \leq\min_{(\mathbf{v},\mathbf{w})\in\mathcal{D}}
  \max_{\mathbf{b}\in\mathcal{S}}
  \left\{
  1-e^{-\varepsilon_k^2}+e^{-\varepsilon_k^2}\phi
  (x_k+\delta^{\varepsilon_k}(\mathbf{v},\mathbf{w},\mathbf{b}))
  \right\}
  +e^{-\varepsilon_k^2}\left\{(u^{\varepsilon_k}-\phi)(x_k)+\varepsilon_k^3\right\}.
\end{equation*}
Thus, we obtain that 
\begin{equation*}
  \overline{u}(x_0)+F_*(D\phi(x_0),D^2\phi(x_0))\leq 1
\end{equation*}
by the same calculation as the case that $x_0\in \Omega$. \\
\textbf{Supersolution test for }$\underline{u}$\textbf{.}
Fix $x_0\in\Omega$ and consider $r>0$ 
and smooth function $\phi$ such that
\begin{equation*}
  \underline{u}(x_0)=\phi(x_0),\ B(x_0,r)\subset\Omega,\ 
  \underline{u}-\phi > 0 \text{ in } B(x_0,r)\setminus\{x_0\}.
\end{equation*}
Similarly to (\ref{eq:game-conv-pf2}), 
there exist sequences $\{x_k\}\subset B(x_0,r)$ and
$\{\varepsilon_k\}$ such that
\begin{equation}\label{eq:game-conv-pf7}
  x_k\to x_0,\ \varepsilon_k\to 0,\ 
  u^{\varepsilon_k}(x_k)\to \underline{u}(x_0),\ 
  (u^{\varepsilon_k}-\phi)(x_k) < \inf_{B(x_0,r)}(u^{\varepsilon_k}-\phi)+\varepsilon_k^3.
\end{equation}
Similarly to (\ref{eq:game-conv-pf3}), we have
\begin{align}\label{eq:game-conv-pf8}
  u^{\varepsilon_k}(x_k)
  &= \inf_{(\mathbf{v},\mathbf{w})\in \mathcal{D}}\max_{\mathbf{b}\in\mathcal{S}}
  \left\{ 1-e^{-\varepsilon_k^2} +e^{-\varepsilon_k^2}
  u^{\varepsilon_k}(x_k+\delta^{\varepsilon_k}(\mathbf{v},\mathbf{w},\mathbf{b}))\right\}\\
  &\geq \min_{(\mathbf{v},\mathbf{w})\in \mathcal{D}}\max_{\mathbf{b}\in\mathcal{S}}
  \left\{ 1-e^{-\varepsilon_k^2} +e^{-\varepsilon_k^2}
  \phi(x_k+\delta^{\varepsilon_k}(\mathbf{v},\mathbf{w},\mathbf{b}))\right\}
  +e^{-\varepsilon_k^2}\left\{(u^{\varepsilon_k}-\phi)(x_k)-\varepsilon_k^3\right\}\notag
\end{align}
for sufficiently large $k$.\\
\textbf{Case} $\mathbf{1'}$\textbf{.} The case that $D\phi(x_0)\neq 0$. 
We noting that $|D\phi(x_k)|$ is bounded 
from below by a positive constant and using 
Lemma \ref{lem:lem-game-conv} (3), 
it follows that 
\begin{equation*}
  \underline{u}(x_0) \geq 1-F(D\phi(x_0),D^2\phi(x_0))
\end{equation*}
by the same calculation as (\ref{eq:game-conv-pf4}).\\
\textbf{Case} $\mathbf{2'}$\textbf{.} The case that $D\phi(x_0)=0$.
Using Lemma \ref{lem:lem-game-conv} (2) for sufficiently large $k$, 
we have
\begin{equation*}
  \max_{\mathbf{b}\in\mathcal{S}}\phi\left(
    x_k+\delta^{\varepsilon_k} (\mathbf{v},\mathbf{w},\mathbf{b})
  \right)
  \geq \phi(x_k)+\varepsilon_k^2 \left\{
    \tr[\sigma(v^1){}^t\!\sigma(v^1)D^2\phi(x_k)]-c(v^1)|D\phi(x_k)|
  \right\}+o(\varepsilon_k^2)
\end{equation*}
for all $v^1\in\partial B(0,1)$. 
Thus, 
\begin{align*}
  u^{\varepsilon_k}(x_k)
  &\geq \min_{(\mathbf{v},\mathbf{w})\in \mathcal{D}}\max_{\mathbf{b}\in\mathcal{S}}
  \left\{ 1-e^{-\varepsilon_k^2} +e^{-\varepsilon_k^2}
  \phi(x_k+\delta^{\varepsilon_k}(\mathbf{v},\mathbf{w},\mathbf{b}))\right\}
  +e^{-\varepsilon_k^2}\left\{(u^{\varepsilon_k}-\phi)(x_k)-\varepsilon_k^3\right\}\\
  &\geq 1-e^{-\varepsilon_k^2}+e^{-\varepsilon_k^2}
  \min_{v^1}\left\{
    \tr [\sigma(v^1){}^t\!\sigma(v^1)D^2\phi(x_k)]-c(v^1)|D\phi(x_k)|
  \right\}\\
  &\quad \quad +e^{-\varepsilon_k^2}u^{\varepsilon_k}(x_k)+o(\varepsilon_k^2).
\end{align*}
And thus, 
\begin{align*}
  &\underline{u}(x_0)+F^*(0,D^2\phi(x_0))\\
  &\geq \limsup_{k\to\infty}\left\{
    \frac{1-e^{-\varepsilon_k^2}}{\varepsilon_k^2}
    u^{\varepsilon_k}(x_k)
    +e^{-\varepsilon_k^2}\max_{v^1}
    \left\{
      -\tr [\sigma(v^1){}^t\sigma(v^1)D^2\phi(x_k)]+c(v^1)|D\phi(x_k)|
    \right\}
  \right\}\\
  &\geq \lim_{k\to\infty}\left\{
    \frac{1-e^{-\varepsilon_k^2}}{\varepsilon_k^2}+o(1)
  \right\}
  =1.
\end{align*}

The super boundary condition of $\underline{u}$ 
is shown by the same way as the sub boundary condition of $\overline{u}$.
\end{proof}

\section{Comparison for solutions and value of the game}

In this section, we provide the proof of our main result, Theorem 1.1.
Since the arguments for the proofs of (1) and (2) cannot be entirely parallel, 
we prove each of them separately.
Through this section, we denote by $u^\varepsilon$ and $U^\varepsilon$,
respectively,
the function defined by 
\begin{equation*}
  u^\varepsilon(x)
  = \inf_{(\mathbf{v},\mathbf{w})\in \mathcal{D}}\max_{\mathbf{b}\in\mathcal{S}}
  \left\{
  \begin{aligned}
    &1-e^{-\varepsilon^2}+e^{-\varepsilon^2}
      u^\varepsilon (x+\delta^\varepsilon(\mathbf{v},\mathbf{w},\mathbf{b}))
      &&\text{if}\ x+\delta^\varepsilon(\mathbf{v},\mathbf{w},\mathbf{b}) \in \R^n\setminus \overline{D}_0,\\
      &1-e^{-\varepsilon^2}+e^{-\varepsilon^2}
      \psi(G (x+\delta^\varepsilon(\mathbf{v},\mathbf{w},\mathbf{b})))
      &&\text{if}\ x+\delta^\varepsilon(\mathbf{v},\mathbf{w},\mathbf{b}) \in \overline{D}_0
  \end{aligned}
  \right.
\end{equation*}
and 
\begin{equation*}
  U^\varepsilon(x)=\psi^{-1}(u^\varepsilon(x)),
\end{equation*}
where $D_0\subset\R^n$ is a given open set 
and $G$ is a given continuous function defined on 
a neighborhood of $\partial D_0$.
Recall that we have defined 
$\psi : \R\cup\{+\infty\}\to (-\infty,1]$ by 
\begin{align*}
  \psi(r) = \left\{
  \begin{aligned}
    &1-e^{-r} &&\text{if } r<+\infty,\\
    &1 &&\text{if } r=+\infty.
  \end{aligned}
  \right.
\end{align*}
Let 
\begin{align*}
  &\overline{u}(x)
  = \limsup_{\substack{\R^n\setminus \overline{D}_0 \ni y\to x \\ \varepsilon \to 0}} u^\varepsilon (y),
  \quad 
  \underline{u}(x)
  = \liminf_{\substack{\R^n\setminus \overline{D}_0 \ni y\to x \\ \varepsilon \to 0}} u^\varepsilon (y),
  \quad x\in\R^n \setminus D_0,\\
  &\overline{U}(x)
  = \limsup_{\substack{\R^n\setminus \overline{D}_0 \ni y\to x \\ \varepsilon \to 0}} U^\varepsilon (y),
  \quad 
  \underline{U}(x)
  = \liminf_{\substack{\R^n\setminus \overline{D}_0 \ni y\to x \\ \varepsilon \to 0}} U^\varepsilon (y),
  \quad x\in\R^n \setminus D_0.
\end{align*}

\subsection{A proof of comparison principle}
We begin with the proof of Theorem 1.1(2).

\begin{proof}[Proof of Theorem $1.1 (2)$]
Let $v(x)=\psi(V(x))$ for $x\in \widetilde{D}\setminus D_0$ 
and $g(x)=\psi(G(x))$ for $x\in \partial D_0$.
Then, $v$ satisfies 
\begin{align*}
  \left\{
  \begin{aligned}
    v+F(Dv,D^2v)&\geq 1 &&\text{in } \widetilde{D}\setminus \overline{D}_0,\\
    v &\geq g &&\text{on }\partial D_0
  \end{aligned}
  \right.
\end{align*}
in the sense of viscosity solutions and 
\begin{align*}
  v< \psi(t) \ \text{in } \widetilde{D}\setminus \overline{D}_0,
  \quad v(x) \to \psi(t)\ \text{as } x\to x_0
  \text{ for all } x_0\in\partial \widetilde{D}
\end{align*}
in the usual sense.
For each small $\varepsilon >0$, 
let $v^\varepsilon\in \mathcal{B}(\widetilde{D}\setminus\overline{D}_0)$ 
be the function defined by 
\begin{align*}
  v^\varepsilon (x)
  = \inf_{(\mathbf{v},\mathbf{w})\in \mathcal{D}}\max_{\mathbf{b}\in\mathcal{S}}
  \left\{
  \begin{aligned}
    &1-e^{-\varepsilon^2}+e^{-\varepsilon^2}
    v^\varepsilon (x+\delta^\varepsilon(\mathbf{v},\mathbf{w},\mathbf{b}))
    &&\text{if}\ x+\delta^\varepsilon(\mathbf{v},\mathbf{w},\mathbf{b}) 
    \in \widetilde{D}\setminus \overline{D}_0,\\
    &1-e^{-\varepsilon^2}+e^{-\varepsilon^2}
    g (x+\delta^\varepsilon(\mathbf{v},\mathbf{w},\mathbf{b}))
    &&\text{if}\ x+\delta^\varepsilon(\mathbf{v},\mathbf{w},\mathbf{b})
    \in \overline{D}_0,\\
    &1-e^{-\varepsilon^2}+e^{-\varepsilon^2}\psi(t)
    &&\text{if}\ x+\delta^\varepsilon(\mathbf{v},\mathbf{w},\mathbf{b})
    \in \R^n\setminus \widetilde{D}
  \end{aligned}
  \right.
\end{align*}
for a continuous extention of $g$ to 
a neighborhood of $\partial D_0$.
By the conclutions of 
Theorem \ref{thm:game-conv} for $v^\varepsilon$ and
Theorem \ref{thm:comparison-Dirichlet},
we have 
\begin{align}\label{eq:comparison-game1-pf4}
  v(x)\geq \limsup_{\substack{
    \widetilde{D}\setminus \overline{D}_0\ni y\to x\\
    \varepsilon\to 0}} v^\varepsilon (y)
  \quad \text{for all } x\in \widetilde{D}\setminus D_0.
\end{align}

Fix $x\in \widetilde{D}\setminus D_0$ 
and take $\alpha>0$ arbitrarily small with $\psi(t)-v(x)>4\alpha>0$.
Then, the inequality
\begin{align}\label{eq:comparison-game1-pf1}
  \psi(t)-2\alpha 
  \geq v(x)+2\alpha \geq \limsup_{\substack{
    \widetilde{D}\setminus \overline{D}_0\ni y\to x\\
    \varepsilon\to 0}} v^\varepsilon (y)+2\alpha 
  \geq v^\varepsilon (x')+\alpha
\end{align}
holds if $\varepsilon$ and $|x'-x|$ are sufficiently small.

Here, we define:
for each $\varepsilon$ and $x'$ satisfying (\ref{eq:comparison-game1-pf1}),
we say the sequence 
$\{(\mathbf{v}_j,\mathbf{w}_j,\mathbf{b}_j)\}_{j=1}^N\subset \mathcal{D}\times\mathcal{S}$
$(N\in\N\cup\{\infty\})$ \textit{follows Player I's strategy} $S^{\text{I}}_v$ 
if for each $1\leq j<N$, we have 
\begin{align*}
  &y_j\in \widetilde{D} \setminus \overline{D}_0,\quad 
  v^\varepsilon(y_{j-1})+\frac{\alpha}{2^{j}}
  >\max_{\mathbf{b}\in\mathcal{S}}\left\{ 1-e^{-\varepsilon^2}+
  e^{-\varepsilon^2}
  \widehat{v}^\varepsilon (y_{j-1} +\delta^\varepsilon(
    \mathbf{v}_{j},\mathbf{w}_{j},\mathbf{b}
  ))\right\},\\
  &\text{and } y_N\not\in \widetilde{D} \setminus \overline{D}_0 
  \text{ if }N<\infty,
\end{align*}
where $\{y_j\}$ is defined by
\begin{align*}
  y_0=x',\quad y_{j}=y_{j-1} + \delta^\varepsilon (
    \mathbf{v}_{j},\mathbf{w}_{j},\mathbf{b}_{j}
  )
\end{align*}
and we set
\begin{align*}
  \widehat{v}^\varepsilon (y) 
  =\left\{
    \begin{aligned}
      &v^\varepsilon (y)
      &&\text{if}\ y\in \widetilde{D}\setminus \overline{D}_0,\\
      &g (y)
      &&\text{if}\ y\in \overline{D}_0,\\
      &\psi(t)
      &&\text{if}\ y\in \R^n\setminus \widetilde{D}.
    \end{aligned}
    \right.
\end{align*}

Then, we have the following claim.\\
\textbf{Claim.} For each $\varepsilon$ and $x'$ satisfying (\ref{eq:comparison-game1-pf1})
and for any $\{(\mathbf{v}_j,\mathbf{w}_j,\mathbf{b}_j)\}_{j=1}^N$
following Player I's strategy $S^{\text{I}}_v$, 
either $N=\infty$ or that $N<\infty$ and $y_N\in \overline{D}_0$.\\
\textit{Proof of Claim.} 
Assume by contradiction that there is a sequence $\{(\mathbf{v}_j,\mathbf{w}_j,\mathbf{b}_j)\}_{j=1}^N$
follows $S^{\text{I}}_v$ and $N<\infty$, $y_N\not\in\overline{D}_0$.
Then, $y_N\in \R^n\setminus \widetilde{D}$ 
and $\widehat{v}^\varepsilon (y_N)=\psi(t)$. 
However, by (\ref{eq:comparison-game1-pf1}), it follows that 
\begin{align*}
  \psi(t)-2\alpha 
  &\geq v^\varepsilon (x')+\alpha 
  = v^\varepsilon(x')+\alpha\sum_{j=1}^\infty \frac{1}{2^{j}}\\
  &> v^\varepsilon(x') +\sum_{j=1}^\infty \frac{\alpha e^{-(j-1)\varepsilon^2}}{2^j}
  =\left(v^\varepsilon(x')+\frac{\alpha}{2}\right)+\frac{\alpha e^{-\varepsilon^2}}{4}
  +\sum_{j=3}^\infty \frac{\alpha e^{-(j-1)\varepsilon^2}}{2^j}\\
  &>1-e^{-\varepsilon^2}+e^{-\varepsilon^2}\left(
    v^\varepsilon(y_1)+\frac{\alpha}{4}
  \right)
  +\frac{\alpha e^{-2\varepsilon^2}}{8}+\sum_{j=4}^\infty 
  \frac{\alpha e^{-(j-1)\varepsilon^2}}{2^j}\\
  &>1-e^{-2\varepsilon^2}+e^{-2\varepsilon^2}
  \left(v^\varepsilon (y_2)+\frac{\alpha}{8}\right)
  +\frac{\alpha e^{-3\varepsilon^2}}{16}
  +\sum_{j=5}^\infty \frac{\alpha e^{-(j-1)\varepsilon^2}}{2^j}\\
  &>\cdots > 1-e^{-N\varepsilon^2} + 
  e^{-N\varepsilon^2}\left(\psi(t)+\frac{\alpha}{2^{N+1}}\right)
  >\psi(N\varepsilon^2 + t) >\psi(t).
\end{align*}
It is a contradiction.

\vspace{\baselineskip}
Next, we define: 
for each $\varepsilon$ and $x'$ with (\ref{eq:comparison-game1-pf1}),
we say the sequence 
$\{(\mathbf{v}_j,\mathbf{w}_j,\mathbf{b}_j)\}_{j=1}^N\subset \mathcal{D}\times\mathcal{S}$
$(N\in\N\cup\{\infty\})$ \textit{follows Player II's strategy} $S^{\text{II}}_u$ 
if for each $1\leq j<N$, we have 
\begin{align*}
  &y_j\in \R^n \setminus \overline{D}_0,\quad 
  \widehat{u}^\varepsilon(y_{j-1}+\delta^\varepsilon(\mathbf{v}_j,\mathbf{w}_j,\mathbf{b}_j)) 
  =\max_{\mathbf{b}\in\mathcal{S}}
  \widehat{u}^\varepsilon (y_{j-1} +\delta^\varepsilon(
    \mathbf{v}_{j},\mathbf{w}_{j},\mathbf{b})),\\
  &\text{and } y_N\in \overline{D}_0 \text{ if }N<\infty,
\end{align*}
where $\{y_j\}$ is defined by
\begin{align*}
  y_0=x',\quad y_{j}=y_{j-1} + \delta^\varepsilon (
    \mathbf{v}_{j},\mathbf{w}_{j},\mathbf{b}_{j}
  )
\end{align*}
and we set
\begin{align*}
  \widehat{u}^\varepsilon (y) 
  =\left\{
    \begin{aligned}
      &u^\varepsilon (y)
      &&\text{if}\ y\in \R^n\setminus \overline{D}_0,\\
      &g (y)
      &&\text{if}\ y\in \overline{D}_0.
    \end{aligned}
    \right.
\end{align*}

For each $\varepsilon$ and $x'$ with (\ref{eq:comparison-game1-pf1}),
we can construct a sequence $\{(\overline{\mathbf{v}}_j,\overline{\mathbf{w}}_j,\overline{\mathbf{b}}_j)\}$
which follows both $S^{\text{I}}_v$ and $S^{\text{II}}_u$ 
as below.

First, choose $(\overline{\mathbf{v}}_1,\overline{\mathbf{w}}_1)\in\mathcal{D}$ such that 
\begin{align*}
  v^\varepsilon(x')+\frac{\alpha}{2}
  >\max_{\mathbf{b}\in\mathcal{S}}\left\{ 1-e^{-\varepsilon^2}+
  e^{-\varepsilon^2}
  \widehat{v}^\varepsilon (x' +\delta^\varepsilon(
    \overline{\mathbf{v}}_{1},\overline{\mathbf{w}}_{1},\mathbf{b}
  ))\right\}
\end{align*}
and then, choose $\overline{\mathbf{b}}_1\in\mathcal{S}$ such that 
\begin{align*}
  \widehat{u}^\varepsilon (x' +\delta^\varepsilon(
    \overline{\mathbf{v}}_{1},\overline{\mathbf{w}}_{1},\overline{\mathbf{b}}_1
  ))
  =\max_{\mathbf{b}\in\mathcal{S}}
  \widehat{u}^\varepsilon (x' +\delta^\varepsilon(
    \overline{\mathbf{v}}_{1},\overline{\mathbf{w}}_{1},\mathbf{b})).
\end{align*}
Set $\overline{y}_1=x'+\delta^\varepsilon(\overline{\mathbf{v}}_{1},\overline{\mathbf{w}}_{1},\overline{\mathbf{b}}_1)$.
At this time, 
if $\overline{y}_1\not\in\widetilde{D}\setminus \overline{D}_0$, 
set $\overline{N}=1$.
If $\overline{y}_1\in\widetilde{D}\setminus \overline{D}_0$,
go to the next step.

For $j=1,2,\ldots$,
we iteratively choose $(\overline{\mathbf{v}}_j,\overline{\mathbf{w}}_j)\in\mathcal{D}$ such that 
\begin{align*}
  v^\varepsilon(\overline{y}_{j-1})+\frac{\alpha}{2^j}
  >\max_{\mathbf{b}\in\mathcal{S}}\left\{ 1-e^{-\varepsilon^2}+
  e^{-\varepsilon^2}
  \widehat{v}^\varepsilon (\overline{y}_{j-1} +\delta^\varepsilon(
    \overline{\mathbf{v}}_{j},\overline{\mathbf{w}}_{j},\mathbf{b}
  ))\right\}
\end{align*}
and choose $\overline{\mathbf{b}}_j\in\mathcal{S}$ such that 
\begin{align*}
  \widehat{u}^\varepsilon (\overline{y}_{j-1} +\delta^\varepsilon(
    \overline{\mathbf{v}}_{j},\overline{\mathbf{w}}_{j},\overline{\mathbf{b}}_j
  ))
  =\max_{\mathbf{b}\in\mathcal{S}}
  \widehat{u}^\varepsilon (\overline{y}_{j-1} +\delta^\varepsilon(
    \overline{\mathbf{v}}_{j},\overline{\mathbf{w}}_{j},\mathbf{b})),
\end{align*}
and then, set $\overline{y}_j=\overline{y}_{j-1}+\delta^\varepsilon(
\overline{\mathbf{v}}_{j},\overline{\mathbf{w}}_{j},\overline{\mathbf{b}}_j)$.
We set $\overline{N}=j$ if $\overline{y}_j\not\in\widetilde{D}\setminus\overline{D}_0$ occurs
for the first time.

The sequence $\{(\overline{\mathbf{v}}_j,\overline{\mathbf{w}}_j,\overline{\mathbf{b}}_j)\}_{j=1}^{\overline{N}}$
we can obtain by this method follows $S^{\text{I}}_v$ 
by its construction and we have that 
\begin{align*}
  &v^\varepsilon (x')+\alpha 
  > 1- e^{-j\varepsilon^2}+e^{-j\varepsilon^2}
  \left(v^\varepsilon(\overline{y}_j)+\frac{\alpha}{2^{j+1}}\right),\\
  &u^\varepsilon(x') \leq 1-e^{-j\varepsilon^2}
  +e^{-j\varepsilon^2} u^\varepsilon (\overline{y}_j)
  \quad (1\leq j <N).
\end{align*}
Moreover, $\{(\overline{\mathbf{v}}_j,\overline{\mathbf{w}}_j,\overline{\mathbf{b}}_j)\}_{j=1}^{\overline{N}}$
also follows $S^{\text{II}}_u$ since we have $\overline{y}_{\overline{N}}\in\overline{D}_0$ if $\overline{N}<\infty$ 
due to the above claim.
Thus, we obtain 
\begin{align}\label{eq:comparison-game1-pf2}
  u^\varepsilon (x')
  \leq 1-e^{-\overline{N}\varepsilon^2}
  +e^{-\overline{N}\varepsilon^2}g(\overline{y}_{\overline{N}})
  < v^\varepsilon (x')+\alpha
\end{align}
if $\overline{N}<\infty$, 
or else 
\begin{align}\label{eq:comparison-game1-pf3}
  u^\varepsilon(x')\leq 1\leq v^\varepsilon (x')+\alpha
\end{align} 
by letting $j\to\infty$ if $\overline{N}=\infty$.

Finally, we have
\begin{align*}
  \overline{U}(x)
  =\limsup_{\substack{
    \widetilde{D}\setminus\overline{D}_0\ni x'\to x\\
    \varepsilon\to 0
  }}\psi^{-1}(u^\varepsilon (x'))
  &\leq \limsup_{\substack{
    \widetilde{D}\setminus\overline{D}_0\ni x'\to x\\
    \varepsilon\to 0
  }}\psi^{-1}(v^\varepsilon (x')+\alpha)
  \leq\psi^{-1}\Bigg(
    \limsup_{\substack{
    \widetilde{D}\setminus\overline{D}_0\ni x'\to x\\
    \varepsilon\to 0
  }}v^\varepsilon (x')+\alpha
  \Bigg)\\
  &\leq \psi^{-1}(v(x)+\alpha)
  \leq V(x)-\log\left(1-\frac{\alpha}{1-v(x)}\right)
\end{align*}
by (\ref{eq:comparison-game1-pf2}), (\ref{eq:comparison-game1-pf3})
and (\ref{eq:comparison-game1-pf1}).
Since this inequality holds for arbitrarily small $\alpha>0$,
we conclude that $\overline{U}(x)\leq V(x)$,
and thus, this implies $\widetilde{D}\subset\{\overline{U}<t\}\cup D_0$.
\end{proof}

Finally, we prove Theorem 1.1(1).

\begin{proof}[Proof of Theorem $1.1(1)$]
Let $w(x)=\psi(W(x))$ for $x\in D\setminus D_0$ 
and $g(x)=\psi(G(x))$ for $x\in \partial D_0$.
Then, $w$ satisfies 
\begin{align*}
  \left\{
  \begin{aligned}
    w+F(Dw,D^2w)&\leq 1 &&\text{in } D\setminus \overline{D}_0,\\
    w &\leq g &&\text{on }\partial D_0
  \end{aligned}
  \right.
\end{align*}
in the sense of viscositysolutions and 
\begin{align}\label{eq:comparison-game2-pf1}
  w< \psi(t) \ \text{in } D\setminus \overline{D}_0,
  \quad w \to \psi(t)\ \text{as } x\to x_0 \text{ for all } x_0\in \partial D
\end{align}
in the usual sense.
For each small $\varepsilon >0$, 
let $w^\varepsilon\in \mathcal{B}(D\setminus\overline{D}_0)$ 
be the function defined by 
\begin{align*}
  w^\varepsilon (x)
  = \inf_{(\mathbf{v},\mathbf{w})\in \mathcal{D}}\max_{\mathbf{b}\in\mathcal{S}}
  \left\{
  \begin{aligned}
    &1-e^{-\varepsilon^2}+e^{-\varepsilon^2}
    w^\varepsilon (x+\delta^\varepsilon(\mathbf{v},\mathbf{w},\mathbf{b}))
    &&\text{if}\ x+\delta^\varepsilon(\mathbf{v},\mathbf{w},\mathbf{b}) 
    \in D\setminus \overline{D}_0,\\
    &1-e^{-\varepsilon^2}+e^{-\varepsilon^2}
    g (x+\delta^\varepsilon(\mathbf{v},\mathbf{w},\mathbf{b}))
    &&\text{if}\ x+\delta^\varepsilon(\mathbf{v},\mathbf{w},\mathbf{b})
    \in \overline{D}_0,\\
    &1-e^{-\varepsilon^2}+e^{-\varepsilon^2}\psi(t)
    &&\text{if}\ x+\delta^\varepsilon(\mathbf{v},\mathbf{w},\mathbf{b})
    \in \R^n\setminus D
  \end{aligned}
  \right.
\end{align*}
for a continuous extention of $g$ to 
a neighborhood of $\partial D_0$.
Similarly to (\ref{eq:comparison-game1-pf4}) in the proof of previous theorem,
we have 
\begin{align*}
  w(x)\leq \liminf_{\substack{
    \widetilde{D}\setminus \overline{D}_0\ni y\to x\\
    \varepsilon\to 0}} w^\varepsilon (y)
  \quad \text{for all } x\in D\setminus D_0.
\end{align*}

First, let us show that $D\supset\{\underline{U}<t\}$.
If $D=\R^n$, it is obvious, 
and therefore we only need to consider the case $\partial D\neq \emptyset$.
Assume by contradiction that 
there exists $x_0\in\{\underline{U}<t\}\setminus D$.
Then, there exist positive constants $\alpha,r >0$ such that 
\begin{align}\label{eq:comparison-game2-pf5}
\left\{
\begin{aligned}
  &\psi (t) - \alpha < w < \psi (t)\text{ in }
  D\cap \{x \mid d(x,\partial D)<r\},\\
  &\underline{u}(x_0)<\psi(t)-4\alpha,\\
  &\{x \mid d(x,D_0)<r\}\subset D
\end{aligned}
\right.
\end{align}
by (\ref{eq:comparison-game2-pf1}) and $\underline{U}(x_0)<t$.
Moreover, there are infinitely many $(\varepsilon,x')$ satisfying 
\begin{align}\label{eq:comparison-game2-pf2}
  u^\varepsilon (x')<\underline{u}(x_0)+\alpha
\end{align} 
on arbitrary neighborhood of $(0,x_0)\in[0,\infty)\times\R^n$
by the definition of $\underline{u}$.

We say that,
for each $(\varepsilon,x')$ satisfying 
(\ref{eq:comparison-game2-pf2}),
the sequence 
$\{(\mathbf{v}_j,\mathbf{w}_j,\mathbf{b}_j)\}_{j=1}^N\subset \mathcal{D}\times\mathcal{S}$
$(N\in\N\cup\{\infty\})$ \textit{follows Player I's strategy} $S^{\text{I}}_u$ 
if for each $1\leq j<N$, we have 
\begin{align*}
  &y_j\in \R^n \setminus \overline{D}_0,\quad 
  u^\varepsilon(y_{j-1})+\frac{\alpha}{2^{j}}
  >\max_{\mathbf{b}\in\mathcal{S}}\left\{ 1-e^{-\varepsilon^2}+
  e^{-\varepsilon^2}
  \widehat{u}^\varepsilon (y_{j-1} +\delta^\varepsilon(
    \mathbf{v}_{j},\mathbf{w}_{j},\mathbf{b}
  ))\right\},\\
  &\text{and } y_N\in \overline{D}_0 
  \text{ if }N<\infty,
\end{align*}
where $\{y_j\}$ is defined by
\begin{align*}
  y_0=x',\quad y_{j}=y_{j-1} + \delta^\varepsilon (
    \mathbf{v}_{j},\mathbf{w}_{j},\mathbf{b}_{j}
  )
\end{align*}
and we set
\begin{align*}
  \widehat{u}^\varepsilon (y) 
  =\left\{
    \begin{aligned}
      &u^\varepsilon (y)
      &&\text{if}\ y\in \R^n \setminus \overline{D}_0,\\
      &g (y)
      &&\text{if}\ y\in \overline{D}_0.
    \end{aligned}
    \right.
\end{align*}
If $\{(\mathbf{v}_j,\mathbf{w}_j,\mathbf{b}_j)\}_{j=1}^N$
follows $S^{\text{I}}_u$, we have $N<\infty$ since
\begin{align}\label{eq:comparison-game2-pf7}
  1-e^{-j\varepsilon^2} + e^{-j\varepsilon^2}
  \left(u^\varepsilon(y_j)+\frac{\alpha}{2^{j+1}}\right)
  <u^\varepsilon(x')+\alpha
  <\underline{u}(x_0)+2\alpha<\psi(t)-2\alpha <1
\end{align}
holds for each $1\leq j <N$;
otherwise we have $1<\psi(t)-2\alpha$ by letting $j\to\infty$,
which is a contradiction.
Therefore, there is an integer 
$\tau=\tau(x',\varepsilon,\{(\mathbf{v}_j,\mathbf{w}_j,\mathbf{b}_j)\}_{j=1}^N)$
determined by
\begin{align}\label{eq:comparison-game2-pf3}
  \tau =\max\{j\in\N \mid d(y_j,\partial D)<r,\ y_k\in D\ \text{for } k\geq j\}.
\end{align}
\textbf{Claim.}
For each $(\varepsilon,x')$ with (\ref{eq:comparison-game2-pf2}),
there exists a sequence 
$\{(\mathbf{v}^0_j,\mathbf{w}^0_j,\mathbf{b}^0_j)\}_{j=1}^{N_0}$
which follows Player I's strategy $S^{\text{I}}_u$ 
satisfying the following property.
\begin{align}\label{eq:comparison-game2-pf4}
  &\text{Let } 
  \{y^0_j\} \text{ be the trajectory determined by }
  \{(\mathbf{v}^0_j,\mathbf{w}^0_j,\mathbf{b}^0_j)\}_{j=1}^{N_0}
  \text{ with } y^0_0=x'\notag\\ 
  &\text{and }
  \tau_0 \text{ be the integer defined by } (\ref{eq:comparison-game2-pf3})
  \text{ for } \{(\mathbf{v}^0_j,\mathbf{w}^0_j,\mathbf{b}^0_j)\}_{j=1}^{N_0}.
  \text{ Then, for all}\notag\\
  &\{(\widetilde{\mathbf{v}}_j,
  \widetilde{\mathbf{w}}_j,
  \widetilde{\mathbf{b}}_j)\}_{j=1}^{\widetilde{N}}
  \text{ with }\notag\\
  &\left\{
  \begin{aligned}
  &\widetilde{y}_j\in D\setminus \overline{D}_0,\\
  &u^\varepsilon(\widetilde{y}_{j-1})+\frac{\alpha}{2^{j+\tau_0}}
  >\max_{\mathbf{b}\in\mathcal{S}}\left\{ 1-e^{-\varepsilon^2}+
  e^{-\varepsilon^2}
  \widehat{u}^\varepsilon (\widetilde{y}_{j-1} 
  +\delta^\varepsilon(
    \widetilde{\mathbf{v}}_{j},\widetilde{\mathbf{w}}_{j},
    \mathbf{b}
  ))\right\},\\
  &\widehat{w}^\varepsilon (\widetilde{y}_{j-1}+\delta^\varepsilon
  (\widetilde{\mathbf{v}}_j,
  \widetilde{\mathbf{w}}_j,
  \widetilde{\mathbf{b}}_j))
  =\max_{\mathbf{b}\in\mathcal{S}}
  \widehat{w}^\varepsilon (\widetilde{y}_{j-1} 
  +\delta^\varepsilon(
    \widetilde{\mathbf{v}}_{j},\widetilde{\mathbf{w}}_{j},
    \mathbf{b}))
    \text{ for } 1\leq j < \widetilde{N},\\
  &\widetilde{y}_{\widetilde{N}}\not \in D\setminus\overline{D}_0 
  \text{ if }\widetilde{N}<\infty,\\
  &\widetilde{y}_0 := y^0_{\tau_0},\ 
  \widetilde{y}_j := \widetilde{y}_{j-1}+\delta^\varepsilon 
  (\widetilde{\mathbf{v}}_j,
  \widetilde{\mathbf{w}}_j,
  \widetilde{\mathbf{b}}_j),
  \end{aligned}
  \right.\\
  &\text{we have } \widetilde{N}<\infty \text{ and } 
  \widetilde{y}_{\widetilde{N}}\in \overline{D}_0.
  \text{ Here, }\widehat{w}^\varepsilon \text{ is defined by}\notag\\
  &\widehat{w}^\varepsilon (y) 
  =\left\{
    \begin{aligned}
      &w^\varepsilon (y)
      &&\text{if}\ y\in D\setminus \overline{D}_0,\\
      &g (y)
      &&\text{if}\ y\in \overline{D}_0,\\
      &\psi(t)
      &&\text{if}\ y\in \R^n\setminus D.
    \end{aligned}
  \right.\notag
\end{align}
\textit{Proof of Claim.}
We argue by contradiction.
Fix $(\varepsilon,x')$ with (\ref{eq:comparison-game2-pf2}) and
assume that for all $\{(\mathbf{v}_j,\mathbf{w}_j,\mathbf{b}_j)\}_{j=1}^{N}$
following strategy $S^{\text{I}}_u$,
there exists a sequence $\{(\widetilde{\mathbf{v}}_j,
\widetilde{\mathbf{w}}_j,
\widetilde{\mathbf{b}}_j)\}_{j=1}^{\widetilde{N}}$
satisfying (\ref{eq:comparison-game2-pf4}) and 
that 
$\widetilde{N}<\infty,\ \widetilde{y}_{\widetilde{N}}\in \R^n\setminus D$
or $\widetilde{N}=\infty$.
We can get a contradiction by configuring a sequence in the following way.

First, fix a sequence $\{(\mathbf{v}^1_j,\mathbf{w}^1_j,\mathbf{b}^1_j)\}_{j=1}^{N_1}$
with strategy $S^{\text{I}}_u$, and then, we get a sequence $\{(\widetilde{\mathbf{v}}_j,
\widetilde{\mathbf{w}}_j,
\widetilde{\mathbf{b}}_j)\}_{j=1}^{\widetilde{N}}$
which satisfies (\ref{eq:comparison-game2-pf4}) and that 
$\widetilde{N}<\infty,\ \widetilde{y}_{\widetilde{N}}\in \R^n\setminus D$
or $\widetilde{N}=\infty$ by the assumption.
Let
\begin{align*}
  &(\mathbf{v}^2_j,\mathbf{w}^2_j,\mathbf{b}^2_j)
  =\left\{
  \begin{aligned}
    &(\mathbf{v}^1_j,\mathbf{w}^1_j,\mathbf{b}^1_j) && (j\leq\tau_1),\\
    &(\widetilde{\mathbf{v}}_{j-\tau_1},
    \widetilde{\mathbf{w}}_{j-\tau_1},
    \widetilde{\mathbf{b}}_{j-\tau_1}) && (\tau_1+1\leq j<\tau_1+1+\widetilde{N}),
  \end{aligned}
  \right.\\
  &y^2_0=x',\ y^2_j= y^2_{j-1}+\delta^\varepsilon (\mathbf{v}^2_j,\mathbf{w}^2_j,\mathbf{b}^2_j)
\end{align*}
where
$\tau_1=\max\{j\in\N \mid d(y^1_j,\partial D)<r, y^1_k\in D \text{ for } k\geq j\}$
and $y^1_0 = x'$, $y^1_j = y^1_{j-1}+\delta^\varepsilon (\mathbf{v}^1_j,\mathbf{w}^1_j,\mathbf{b}^1_j)$.
Here, if $\widetilde{N}=\infty$, 
then $\{(\mathbf{v}^2_j,\mathbf{w}^2_j,\mathbf{b}^2_j)\}_{j=1}^{\infty}$
follows Player I's strategy $S^{\text{I}}_u$.
However, it implies $\widetilde{N}<\infty$ and 
$\widetilde{y}_{\widetilde{N}}=y^2_{\tau_1+\widetilde{N}}\in \overline{D}_0$ 
by (\ref{eq:comparison-game2-pf7}), which contradicts (\ref{eq:comparison-game2-pf4}).
Thus, we have $\widetilde{N}<\infty$ and $\widetilde{y}_{\widetilde{N}}\in\R^n\setminus D$.
Then, we extend the sequence $\{(\mathbf{v}^2_j,\mathbf{w}^2_j,\mathbf{b}^2_j)\}_{j=1}^{\tau_1+\widetilde{N}}$ 
so that it follows $S^{\text{I}}_u$ by choosing 
$(\mathbf{v}^2_j,\mathbf{w}^2_j)\in\mathcal{D}$ such that 
\begin{align*}
  u^\varepsilon(y_{j-1})+\frac{\alpha}{2^{j}}
  >\max_{\mathbf{b}\in\mathcal{S}}\left\{ 1-e^{-\varepsilon^2}+
  e^{-\varepsilon^2}
  \widehat{u}^\varepsilon (y_{j-1} +\delta^\varepsilon(
    \mathbf{v}_{j},\mathbf{w}_{j},\mathbf{b}
  ))\right\}
\end{align*}
and $\mathbf{b}^2_j\in\mathcal{S}$ arbitrarily 
for $j\geq\tau_1+\widetilde{N}+1$ 
until its trajectory $y^2_j=y^2_{j-1}+\delta^\varepsilon 
(\mathbf{v}^2_j,\mathbf{w}^2_j,\mathbf{b}^2_j)$ reaches $\overline{D}_0$.
By (\ref{eq:comparison-game2-pf7}), 
there is an integer $\tau_2$ determined by 
$\tau_2=\max\{j\in\N \mid d(y^2_j,\partial D)<r, y^2_k\in D \text{ for } k\geq j\}$
and $\tau_1 < \tau_1 + \widetilde{N} < \tau_2$ holds 
since
$y^2_{\tau_1+\widetilde{N}}=\tilde{y}_{\tilde{N}}\in\R^n\setminus D$.

For each $l=2,3,\ldots$ and 
$\{(\mathbf{v}^{l}_j,\mathbf{w}^{l}_j,\mathbf{b}^{l}_j)\}_{j=1}^{N_{l}}$, 
we get a sequence $\{(\widetilde{\mathbf{v}}_j,\widetilde{\mathbf{w}}_j,
\widetilde{\mathbf{b}}_j)\}_{j=1}^{\widetilde{N}}$ 
satisfying (\ref{eq:comparison-game2-pf4}) and that
$\widetilde{N}<\infty,\ \widetilde{y}_{\widetilde{N}}\in \R^n\setminus D$ 
or $\widetilde{N}=\infty$
by the assumption, 
and then, let 
\begin{align*}
  &(\mathbf{v}^{l+1}_j,\mathbf{w}^{l+1}_j,\mathbf{b}^{l+1}_j)
  =\left\{
  \begin{aligned}
    &(\mathbf{v}^l_j,\mathbf{w}^l_j,\mathbf{b}^l_j) && (j\leq\tau_l),\\
    &(\widetilde{\mathbf{v}}_{j-\tau_l},
    \widetilde{\mathbf{w}}_{j-\tau_l},
    \widetilde{\mathbf{b}}_{j-\tau_l}) && (\tau_l+1\leq j<\tau_l+1+\widetilde{N}),
  \end{aligned}
  \right.\\
  &y^{l+1}_0=x',\ y^{l+1}_j= y^{l+1}_{j-1}+\delta^\varepsilon 
  (\mathbf{v}^{l+1}_j,\mathbf{w}^{l+1}_j,\mathbf{b}^{l+1}_j).
\end{align*}
By the same argument of the step of $l=1$, 
we have $\widetilde{N}<\infty,\ \widetilde{y}_{\widetilde{N}}\in \R^n\setminus D$
and we can extend 
$\{(\mathbf{v}^{l+1}_j,\mathbf{w}^{l+1}_j,\mathbf{b}^{l+1}_j)\}$
so that it follows  $S^{\text{I}}_u$
until its trajectory $y^{l+1}_j$ reaches $\overline{D}_0$.
By (\ref{eq:comparison-game2-pf7}) again,
there exists an integer 
$\tau_{l+1}:=\max\{j\in\N \mid d(y^{l+1}_j,\partial D)<r, y^{l+1}_k\in D \text{ for } k\geq j\}$
and it satisfies $\tau_l < \tau_{l+1}$.

For all sequences 
$\{(\mathbf{v}^{l}_j,\mathbf{w}^{l}_j,\mathbf{b}^{l}_j)\}_{j=1}^{N_{l}}$
determined inductively as above,
we set 
\begin{align*}
  &(\mathbf{v}^\infty_j, \mathbf{w}^\infty_j,\mathbf{b}^\infty_j)
  =\left\{
  \begin{aligned}
    &(\mathbf{v}^1_j,\mathbf{w}^1_j,\mathbf{b}^1_j)
    &&\text{if } 1\leq j<\tau_{1},\\
    &(\mathbf{v}^{l+1}_j,\mathbf{w}^{l+1}_j,\mathbf{b}^{l+1}_j)
    &&\text{if } \tau_l\leq j<\tau_{l+1},
  \end{aligned}
  \right.\\
  &y^\infty_0 =x',\quad
  y^\infty_j = y^\infty_{j-1}+\delta^\varepsilon (
    \mathbf{v}^\infty_j, \mathbf{w}^\infty_j,\mathbf{b}^\infty_j).
\end{align*}
Then, $\{(\mathbf{v}^{\infty}_j,\mathbf{w}^{\infty}_j,\mathbf{b}^{\infty}_j)\}_{j=1}^{\infty}$
follows $S^{\text{I}}_u$
but its trajectory $y^\infty_j$ cannot 
reach $\overline{D}_0$ in a finite number of steps.
It contradicts (\ref{eq:comparison-game2-pf7}).
Thus, we conclude that the claim holds.
\vspace{\baselineskip}

For each $(\varepsilon,x')$ with (\ref{eq:comparison-game2-pf2}),
we construct a sequence
$\{(\overline{\mathbf{v}}_j,\overline{\mathbf{w}}_j,
\overline{\mathbf{b}}_j)\}_{j=1}^{\overline{N}}$
as below:
take $\{(\mathbf{v}^0_j,\mathbf{w}^0_j,\mathbf{b}^0_j)\}_{j=1}^{N_0}$
satisfying the above claim and let 
\begin{align*}
  (\overline{\mathbf{v}}_j,\overline{\mathbf{w}}_j,
\overline{\mathbf{b}}_j)=
(\mathbf{v}^0_j,\mathbf{w}^0_j,\mathbf{b}^0_j)
\end{align*}
and set $\overline{y}_j=y^0_j$ for $1\leq j \leq \tau_0-1$.
For $j\geq \tau_0$, choose $(\overline{\mathbf{v}}_j,\overline{\mathbf{w}}_j)\in\mathcal{D}$ such that 
\begin{align*}
  u^\varepsilon(\overline{y}_{j-1})+\frac{\alpha}{2^{j}}
  >\max_{\mathbf{b}\in\mathcal{S}}\left\{ 1-e^{-\varepsilon^2}+
  e^{-\varepsilon^2}
  \widehat{u}^\varepsilon (\overline{y}_{j-1} +\delta^\varepsilon(
    \overline{\mathbf{v}}_{j},\overline{\mathbf{w}}_{j},\mathbf{b}
  ))\right\}
\end{align*}
and $\overline{\mathbf{b}}_j\in\mathcal{S}$ such that 
\begin{align*}
  \widehat{w}^\varepsilon (\overline{y}_{j-1} +\delta^\varepsilon(
    \overline{\mathbf{v}}_{j},\overline{\mathbf{w}}_{j},\overline{\mathbf{b}}_j
  ))
  =\max_{\mathbf{b}\in\mathcal{S}}
  \widehat{w}^\varepsilon (\overline{y}_{j-1} +\delta^\varepsilon(
    \overline{\mathbf{v}}_{j},\overline{\mathbf{w}}_{j},\mathbf{b})),
\end{align*}
and set $\overline{y}_j=\overline{y}_{j-1}+\delta^\varepsilon(
\overline{\mathbf{v}}_{j},\overline{\mathbf{w}}_{j},\overline{\mathbf{b}}_j)$ iteratively.
We set $\overline{N}$ as the smallest index $j$ such that $\overline{y}_j\not \in D\setminus \overline{D}_0$ and $j>\tau_0$.

Since the sequence $(\widetilde{\mathbf{v}}_j,
\widetilde{\mathbf{w}}_j,
\widetilde{\mathbf{b}}_j)
:=(\overline{\mathbf{v}}_{j+\tau_0},\overline{\mathbf{w}}_{j+\tau_0},
\overline{\mathbf{b}}_{j+\tau_0})$
satisfies (\ref{eq:comparison-game2-pf4}), 
we have that $\overline{N}<\infty$ and 
$\overline{y}_{\overline{N}}\in \overline{D}_0$ by the above claim.
Therefore, we can calculate 
\begin{align}\label{eq:comparison-game2-pf6}
  w^\varepsilon (y^0_{\tau_0})=
  w^\varepsilon (\overline{y}_{\tau_0})
  &\leq 1-e^{-\varepsilon^2} + e^{-\varepsilon^2} w^\varepsilon (\overline{y}_{\tau_0+1})\\
  &\leq \cdots \leq 1-e^{-\overline{N}\varepsilon^2}
  +e^{-\overline{N}\varepsilon^2} g(\overline{y}_{\overline{N}})\notag\\
  &= 1-e^{-(\overline{N}-1)\varepsilon^2}+e^{-(\overline{N}-1)\varepsilon^2}
  \left(1-e^{-\varepsilon^2} + e^{-\varepsilon^2} g(\overline{y}_{\overline{N}})\right)\notag\\
  &< 1-e^{-(\overline{N}-1)\varepsilon^2}+e^{-(\overline{N}-1)\varepsilon^2}
  u^\varepsilon (\overline{y}_{\overline{N}-1})
  +\frac{\alpha e^{-(\overline{N}-1)\varepsilon^2}}{2^{\overline{N}}}\notag\\
  &<\cdots < 1- e^{-\tau_0 \varepsilon^2} +e^{-\tau_0 \varepsilon^2}
  u^\varepsilon (\overline{y}_{\tau_0}) + \sum_{j=\tau_0}^{\overline{N}-1}
  \frac{\alpha e^{-j\varepsilon^2}}{2^{j+1}}\notag\\
  &<\cdots < 1-e^{-\varepsilon^2} + e^{-\varepsilon^2}
  u^\varepsilon (\overline{y}_1)+\sum_{j=1}^{{\overline{N}-1}}
  \frac{\alpha e^{-j\varepsilon^2}}{2^{j+1}}\notag\\
  &< u^\varepsilon (\overline{y}_0)+\sum_{j=0}^{\overline{N}-1}
  \frac{\alpha e^{-j\varepsilon^2}}{2^{j+1}}
  <u^\varepsilon (x') + \alpha.\notag
\end{align}
Thus, letting $(\varepsilon , x')\to (0,x_0)$ 
while $(\varepsilon,x')$ satisfying (\ref{eq:comparison-game2-pf2}),
we have $u^\varepsilon (x')\to \underline{u}(x_0)$ and 
$y^0_{\tau_0}\to \exists y_0\in \overline{D \cap\{d(\cdot,\partial D)< r\}}$ 
up to subsequences and 
\begin{align*}
  \psi(t)-\alpha &\leq w(y_0)
  \leq \liminf_{\substack{\varepsilon'\to 0\\ y'\to y_0}} w^{\varepsilon'} (y')
  \leq \liminf_{\substack{\varepsilon\to 0\\ x'\to x_0}} w^\varepsilon (y^0_{\tau_0})\\
  &\leq \lim_{\substack{\varepsilon\to 0\\ x'\to x_0}}
  u^{\varepsilon}(x')+\alpha
  =\underline{u}(x_0)+\alpha < \psi(t)-2\alpha
\end{align*}
by (\ref{eq:comparison-game2-pf5}), (\ref{eq:comparison-game2-pf6})
and Theorem \ref{thm:comparison-Dirichlet}.
This is a contradiction.

Finally, we will show that $W\leq \underline{U}$ on $\{\underline{U}<t\}\setminus D_0$.
Fix $x\in \{\underline{U}<t\}\setminus D_0$ and arbitrarily small $\alpha >0$
and take sequences $\{\varepsilon_k\}$ and $\{x_k\}$ such that 
$\varepsilon_k\to 0$, $x_k\to x$ and $u^{\varepsilon_k}(x_k)\to \underline{u}(x)$.
Then, we can show that 
\begin{align*}
  w^{\varepsilon_k} (x_k) \leq u^{\varepsilon_k}(x_k)+\alpha 
\end{align*}
for all $k$ by the same argument of the proof of (\ref{eq:comparison-game1-pf2}) in the previous theorem
since we have already known that $x\in \{\underline{U}<t\}\subset D$.
Thus, it follows that 
\begin{align*}
  \underline{u}(x)+\alpha = \lim_{k\to\infty} u^{\varepsilon_k}(x_k)+\alpha
  \geq \liminf_{k\to\infty } w^{\varepsilon_k}(x_k)
  \geq \liminf_{\substack{\varepsilon\to 0\\ y\to x}}w^\varepsilon (y)
  \geq w(x)
\end{align*}
by Theorem \ref{thm:comparison-Dirichlet}
and this implies $\underline{U}(x)\geq W(x)$.
\end{proof}

\subsection{Applications}

In this subsection, we consider only the case that $G=0$ for (\ref{FBP})
and assume that the function $c$ satisfies
\begin{align}\label{assum:c-positive}
  c(\mathbf{n})>0 \quad\text{for all } \mathbf{n}\in\partial B(0,1)
\end{align}
and $D_0$ is a bounded domain.
In this case, 
following Proposition \ref{prop:evaluating-above} gives a sufficient condition
on the domain $D_0$ 
for the functions $\overline{U}$ and $\underline{U}$ to have a ``meaningful'' values.
Otherwise, in general, $\overline{U}$ and $\underline{U}$ 
can be identically $+\infty$.

\begin{prop}\label{prop:evaluating-above}
  Assume $(\mathrm{A}\ref{assum:homogenity})$-$(\mathrm{A}\ref{assum:rank})$ and $(\ref{assum:c-positive})$.
  Then, there exists $R=R(\sigma,c)>0$ such that 
  for any open set $D_0$ with $D_0\supset B(0,R)$,
  we have $\overline{U}(x),\ \underline{U}(x)<\infty$ for all $x\in \R^n\setminus D_0$.
\end{prop}

\begin{proof}
Let $c_0:=\min_{|\mathbf{n}|=1}|c(\mathbf{n})|$ and  
$C_0:=\max_{|\mathbf{n}|=1}\lVert\sigma(\mathbf{n})\rVert$.
Set
\begin{align*}
  R:=\frac{2C_0^2}{c_0}+1
\end{align*}
and assume $D_0\supset B(0,R)$.
For any $x\in\R^n\setminus \overline{D}_0$ and any small $\varepsilon>0$,
we consider the sequence $\{(\mathbf{v}_j,\mathbf{w}_j,\mathbf{b}_j)\}_{j=1}^N$
constructed as follows:
let $y_0=x$. For each $j=1,2,\ldots$, we define
\begin{align}\label{eq:consentric}
  \mathbf{v}_j:=\left(-\frac{y_{j-1}}{|y_{j-1}|},-\frac{y_{j-1}}{|y_{j-1}|}\right)\in \R^n\times \R^n,\
  \mathbf{w}_j:=\left(e_1,\ldots,e_m\right)\in\R^{m\times m},
\end{align}
where $\{e_1,\ldots,e_m\}$ is a canonical basis of $\R^m$. 
For these $\mathbf{v}_j$ and $\mathbf{w}_j$,
choose $\mathbf{b}_j\in\mathcal{S}$ so that it satisfies
\begin{align*}
  u^\varepsilon (y_{j-1}+\delta^\varepsilon(\mathbf{v}_j,\mathbf{w}_j,\mathbf{b}_j))
  =\max_{\mathbf{b}\in\mathcal{S}}
  \left\{
  \begin{aligned}
    &u^\varepsilon (y_{j-1}+\delta^\varepsilon(\mathbf{v}_j,\mathbf{w}_j,\mathbf{b}))
    &&\text{if }\delta^\varepsilon(\mathbf{v}_j,\mathbf{w}_j,\mathbf{b})\in \R^n\setminus\overline{D}_0,\\
    &0
    &&\text{if }\delta^\varepsilon(\mathbf{v}_j,\mathbf{w}_j,\mathbf{b})\in \overline{D}_0
  \end{aligned}
  \right.
\end{align*}
and let $y_j:=y_{j-1}+\delta^\varepsilon (\mathbf{v}_j,\mathbf{w}_j,\mathbf{b}_j)$.
We define $N$ as the smallest index $j$ such that $y_j\in\overline{D}_0$.
If $y_j\in\R^n\setminus\overline{D}_0$ always holds, we set $N=\infty$.

Then, by (\ref{eq:consentric}) and (A\ref{assum:rank}), 
for each $1\leq j <N$, we have
\begin{align*}
  |y_j|^2
  &=\left|y_{j-1}+\sqrt{2}\varepsilon\sum_{i=1}^m b_j^i\sigma (y_{j-1})e_i 
  -\varepsilon^2 c\left(\frac{y_{j-1}}{|y_{j-1}|}\right)\frac{y_{j-1}}{|y_{j-1}|}\right|^2\\
  &=\left||y_{j-1}|-\varepsilon^2c\left(\frac{y_{j-1}}{|y_{j-1}|}\right)\right|^2
  +2\varepsilon^2 \left|\sum_{i=1}^m b_j^i\sigma(y_{j-1})e_i\right|^2\\
  &\leq \left||y_{j-1}|-\varepsilon^2 c_0\right|^2 + 2\varepsilon^2 C_0^2.
\end{align*}
Since it holds that
\begin{align*}
  r\geq R = \frac{2C_0^2}{c_0}+1\quad \Rightarrow\quad
  \left|r-\varepsilon^2 c_0\right|^2 + 2\varepsilon^2 C_0^2\leq \left(r- \frac{\varepsilon^2 c_0}{2}\right)^2
\end{align*}
for small $\varepsilon>0$ and 
$y_j\not\in B(0,R)$ for each $1\leq j<N$, we have
\begin{align}\label{eq:capturing-step}
  |y_j|\leq \sqrt{\left||y_{j-1}|-\varepsilon^2 c_0\right|^2 + 2\varepsilon^2 C_0^2}
  \leq |y_{j-1}|-\frac{\varepsilon^2c_0}{2}
  \quad \text{for } 1\leq j <N.
\end{align}
Inequality (\ref{eq:capturing-step}) implies that 
the trajectory $y_j$ can reach $\overline{D}_0$ at most 
$\left[\left(|x|-R\right)/\frac{\varepsilon^2c_0}{2}\right]$ steps.
Therefore, the dynamic programming principle for $U^\varepsilon$ implies
\begin{align*}
  U^\varepsilon (x) < \varepsilon^2 \cdot \frac{|x|-\frac{2C_0^2}{c_0}}{\frac{\varepsilon^2c_0}{2}}
  =\frac{2}{c_0}|x| -\frac{4C_0^2}{c_0^2}.
\end{align*}
Thus, we obtain that
\begin{align}\label{eq:evaluating-above}
  \underline{U}(x)\leq \overline{U}(x)\leq \frac{2}{c_0}|x| -\frac{4C_0^2}{c_0^2}<\infty
  \quad \text{for every } x\in\R^n\setminus D_0.
\end{align}
\end{proof}
We obtain following Proposition \ref{prop:global-sol}
directly from (\ref{eq:evaluating-above}).
\begin{prop}\label{prop:global-sol}
  Assume $(\mathrm{A}\ref{assum:homogenity})$-$(\mathrm{A}\ref{assum:rank})$ and $(\ref{assum:c-positive})$.
  Let $R=R(\sigma,c)>0$ be the constant satisfying Proposition $\ref{prop:evaluating-above}$.
  If $D_0\supset B(0,R)$ holds, then $(\R^n,\overline{U})$ and $(\R^n,\underline{U})$
  are viscosity subsolution and supersolution of $(\mathrm{FBP}_\infty)$, respectively.
\end{prop}

Next, we consider evaluating the large time behavior of 
the domain $D_t$ satisfying (\ref{FBP})
by using a parallel argument of \cite[Section 3]{MR1279982} and the Wulff shape of c: 
\begin{align*}
  \mathop{\mathrm{Wulff}}(c)
  :=\{ x \mid x\cdot \mathbf{n}\leq c(\mathbf{n})\text{ for all } \mathbf{n}\in\partial B(0,1)\}.
\end{align*}
We define
\begin{align}\label{eq:dfn-W}
  W(x):=\max_{|\mathbf{n}|=1} \frac{x\cdot \mathbf{n}}{c(\mathbf{n})}.
\end{align}
Then, we can check that for each $t>0$, the function $W$ defined by (\ref{eq:dfn-W}) satisfies
\begin{align*}
  \{x\mid W(x)\leq t\}
  &=\{x \mid x\cdot \mathbf{n}\leq t c(\mathbf{n}) \text{ for all }\mathbf{n}\in\partial B(0,1)\}\\
  &=\{t x \mid x\cdot \mathbf{n}\leq c(\mathbf{n}) \text{ for all }\mathbf{n}\in\partial B(0,1)\}\\
  &=t \mathop{\mathrm{Wulff}}(c)
\end{align*}
and we have the following proposition.
\begin{prop}\label{prop:Wulff-subsol}
The function $W$ defined by $(\ref{eq:dfn-W})$ satisfies 
$F(DW,D^2 W) \leq 1$ in $\R^n$ in the sense of viscosity solutions.
\end{prop}

\begin{proof}
By the proof of \cite[Proposition 3.4]{MR1279982}, 
it follows that $W$ is a convex and positively 1-homogenous function, 
and satisfies $c(DW)\leq 1$ in $\R^n$.
Assume $(p,X)\in J^{2,+}W(x)$.
Then, we have that $W$ is twice differentiable at $x$ and $p=DW(x)\neq 0$, $X=D^2W(x)\geq O$
from these facts.
Therefore, we conclude $F(p,X)\leq F(p,O)=c(p)\leq 1$.
\end{proof}
Since we assumed that $D_0$ is bounded, 
there exists $t_0>0$ such that 
\begin{align*}
  \overline{D}_0 \subset t_0 \mathop{\mathrm{Wulff}}(c).
\end{align*}
For this $t_0$, we obtain the following proposition from Proposition \ref{prop:Wulff-subsol}.
\begin{prop}
  For each $t\in (0,\infty)$, the pair 
  $(\mathop{\mathrm{int}}((t+t_0)\mathop{\mathrm{Wulff}}(c)),W-t_0)$
  is a viscosity subsolution of $(\mathrm{FBP}_t)$.
\end{prop}

We obtain from the above proposition, 
Theorem \ref{thm:main}(1) and the definition of $\overline{U}$ and $\underline{U}$,
\begin{align*}
  \{\overline{U}<t\}\subset\{\underline{U}<t\}\subset (t+t_0)\mathop{\mathrm{Wulff}}(c)
\end{align*}
for each $t>0$.
By dividing the three sides by $t$, we have
\begin{align*}
  \frac{1}{t}\{\overline{U}<t\}
  \subset \frac{1}{t}\{\underline{U}<t\}
  \subset \frac{t+t_0}{t}\mathop{\mathrm{Wulff}}(c),
\end{align*}
and this implies
\begin{align*}
  \limsup_{t\to\infty}\frac{1}{t}\{\overline{U}<t\}
  \subset
  \limsup_{t\to\infty}\frac{1}{t}\{\underline{U}<t\}
  \subset \mathop{\mathrm{Wulff}}(c).
\end{align*}

\section*{Acknowledgements}
The author would like to thank his supervisor, 
Professor Hiroyoshi Mitake for 
providing him with invaluable insights and 
direction of the study.
The author is grateful to Professors Yoshikazu Giga, Qing Liu
and Hung Vinh Tran for
their helpful advices and kind comments.

\bibliographystyle{alpha}

\end{document}